\documentclass[11pt,reqno,final]{article}
\usepackage{otitis}
\hypersetup{pdfpagemode=UseOutlines,colorlinks=false,pdfpagelayout=SinglePage,pdfstartview=FitH,bookmarksopen=true}

\begin{document}

\title{Dirac generating operators and Manin triples}
\author{\textsc{Zhuo Chen} \\
{\small Peking University, Department of Mathematics and LMAM} \\
{\small Beijing 100871, China} \\
{\small \href{mailto:chenzhuott@gmail.com}{\texttt{chenzhuott@gmail.com}}} \and \textsc{Mathieu Sti\'enon} \\
{\small Pennsylvania State University, Department of Mathematics} \\ {\small 109 McAllister Building, University Park, PA 16802, U.S.A.} \\
{\small \href{mailto:stienon@math.psu.edu}{\texttt{stienon@math.psu.edu}}} }

\date{}
\maketitle

\begin{abstract}
Given a pair of (real or complex) Lie algebroid structures on a vector bundle $A$ (over $M$) and its dual $A^*$, and a line bundle $\module$ such that $\module\otimes\module=(\wedge^{\TOP} A^*\otimes\wedge^{\TOP} T^*M)$, there exist two canonically defined differential operators $\bdees$ and $\bdel$ on  $\sections{\wedge A\otimes\module}$. We prove that the pair $(A,A^*)$ constitutes a Lie bialgebroid if, and only if, the square of $\bdirac =\bdees+\bdel$ is the multiplication by a function on $M$.
As a consequence, we obtain that the pair $(A,A^*)$ is a Lie bialgebroid if, and only if, $\bdirac$ is a  Dirac generating operator as defined by Alekseev \& Xu \cite{AlekseevXu}. Our approach is to establish a list of new identities relating the Lie algebroid structures on $A$ and $A^*$ (Theorem~\ref{Thm:C}).
\end{abstract}

\tableofcontents

\section{Introduction}

A Lie bialgebroid, as introduced by Mackenzie \& Xu,  
is a pair of Lie algebroids $(A, A^*)$ satisfying some compatibility condition \cite{MR1262213, MR1362125}.
They appear naturally in many places in Poisson geometry.
In \cite{math/9910078}, Roytenberg proved that the Lie bialgebroid compatibility condition is equivalent to the equation $\{H, H\}=0$, where $H$ is a certain Hamitonian function on the super-symplectic manifold $T^*A[1]$. The main purpose of this paper is to prove a quantum analogue of this condition.

More precisely, consider a pair of (real or complex) Lie algebroid structures on a vector bundle $A$ and its dual $A^*$. Assume that the line bundle (real or complex) 
$\module=(\wedge^{top} A^*\otimes\wedge^{top} T^*M)^{\thalf}$ exists. Then $\module$ is a module over $A^*$, as discovered by Evens, Lu \& Weinstein \cite{MR1726784}, and the Lie algebroid structures of $A^*$ and $A$ induce two natural differential
operators $\bdees:\sections{\wedge^k A\otimes\module}\to
\sections{\wedge^{k+1} A\otimes\module}$ and
$\bdel:\sections{\wedge^k A\otimes\module} \to\sections{\wedge^{k-1} A\otimes\module}$ (see Equations \eqref{6} to \eqref{11}). Let $\bdirac$ be the sum $\bdees+\bdel$. 

Our main theorem can be summarized as follows:

\begin{thm}\label{thm:A}
The pair $(A,A^*)$ is a Lie bialgebroid over $M$ if, and only if, the operator $\bdirac^2$ is the multiplication by a function $\fsmile\in\cinf{M}$.
\end{thm}

Since the problem is local in nature, we may assume the 
existence of a volume form $s$ on $M$ and a pair of nowhere vanishing top degree forms $\Omega\in\sections{\wedge^{top} A^*}$ and $V\in\sections{\wedge^{top} A}$ dual to each other. Hence one can consider the modular cocycles $\xi_0$ and $X_0$ associated to the Lie algebroids $A$ and $A^*$ respectively \cite{MR1726784}.

A simple computation shows that, for all $u\in\sections{\wedge A}$,
\[ \bdirac^2 u = \big(\thalf (\ld{X_0}+\ld{\xi_0})-\lap\big) u +\thalf\big(\thalf\ip{\xi_0}{X_0}-\del X_0\big) u ,\]
where $\del$ is the BV-operator of the Lie algebroid $A$ \cite{MR1675117} $\OO$-dual to the Lie algebroid  differential $\dee$ (see \eqref{13bis}) and $\lap$ is the ``Laplacian'' $\dees\del+\del\dees$.
Therefore, the problem reduces to prove 
the following

\begin{thm}\label{thm:B}
For a pair of Lie algebroids $(A,A^*)$, $\lap=\thalf(\ld{X_0}+\ld{\xi_0})$ if, and only if, $(A,A^*)$ is a Lie bialgebroid.
\end{thm}

Note that the Laplacian $\lap$ has already appeared in a variety of contexts. For example, the Laplacian of the Lie bialgebroid $TM\oplus (T^*M)_{\pi}$ associated to a Poisson manifold $(M,\pi)$ was fruitfully exploited by Evens \& Lu to compute the cohomology of flag varieties \cite{MR1680047}.

Part of the motivation behind this work is to better understand the Dirac generating operators of the Courant algebroids associated to Manin triples.

Courant algebroids were introduced in \cite{MR1472888} as a way to merge the concept of Lie bialgebra and the bracket on
$\XX(M)\oplus\OO^1(M)$ --- here $M$ is a smooth manifold --- first discovered by Courant \cite{MR998124}. Roytenberg gave an equivalent definition phrased in terms of the Dorfman bracket \cite{math/9910078}, which highlighted the relation of Courant algebroids to $L_{\infty}$-algebras as was observed by Roytenberg \& Weinstein \cite{MR1656228}.

Kosmann-Schwarzbach's derived brackets \cite{MR2104437} provide another way of thinking of Courant algebroids.
Roytenberg translated them in symplectic supermanifold language to study Courant algebroids \cite{math/9910078,MR1958835}. They also motivated the following construction \cite{AlekseevXu}.

Let $E\xto{\pi}M$ be a vector bundle endowed with a non degenerate pseudo-metric $\ip{\cdot}{\cdot}$ on its fibers, and let $\clifford{E}\to M$ be the associated bundle of Clifford algebras. Assume there exists a bundle of Clifford modules $S\to M$, i.e. a vector bundle whose fiber is the Clifford module of the fiber of $\clifford{E}$. The natural $\ZZ_2$-grading of $\sections{\clifford{E}}$ induces a $\ZZ_2$-grading on the operators on $S$. The multiplication by a function $f\in\cinf{M}$ is an even operator while the Clifford action of a section $e\in\sections{E}$ is odd.
{\em A Dirac generating operator}, according to
 Alekseev \& Xu \cite{AlekseevXu}, 
 is an odd operator $\dirac$ on
$\sections{S}$ satisfying the following properties:
\begin{itemize}
\item For all $f\in\cinf{M}$, the operator $\lb{\dirac}{f}$ is the Clifford action of some section of $E$.
\item For all $e_1,e_2\in\sections{E}$, the operator
$\lb{\lb{\dirac}{e_1}}{e_2}$ is the Clifford action of some section of $E$.
\item The square of $\dirac$ is the multiplication by some function on $M$.
\end{itemize}

If $\dirac$ is a generating operator, then the derived bracket
 $\db{e_1}{e_2}=\lb{\lb{\dirac}{e_1}}{e_2}$ on
 $\sections{E}$ --- here $\lb{\cdot}{\cdot}$ stands for the graded
 commutator on the space of graded operators on $\sections{S}$ ---  together
 with the bundle map $\rho:E\to TM$ given by $\rho(e)f=2\ip{\lb{\dirac}{f}}{e}$ endow $(E,\ip{\cdot}{\cdot})$ with a Courant algebroid structure.

In \cite{AlekseevXu}, Alekseev \& Xu have proposed a construction of Dirac generating operators valid for arbitrary Courant algebroids using Courant algebroid connections.
When $E=A\oplus A^*$, Kosmann-Schwarzbach considered
the so called {\em deriving operators}, which generate
the Courant algebroid in the above sense, without requiring that the square of $\dirac$ be the multiplication by some function \cite{MR2103012}. In particular, for a Lie bialgebroid, she proved by a direct argument that $\bdirac=\bdees+\bdel$ (in a slightly different form) is a deriving operator.
Combining Kosmann-Schwarzbach's result with Theorem~\ref{thm:A}, we immediately obtain the following

\begin{thm}\label{thm:C} 
Given a pair of Lie algebroids $(A,A^*)$, it   is a Lie bialgebroid if, and only if, the operator $\bdirac=\bdees+\bdel$ is a Dirac generating operator for the vector bundle $A\oplus A^*$ with the inner product \eqref{3}.
\end{thm}

Although the sufficient condition is just the particularization of Alekseev \& Xu's result to the case of Lie bialgebroids, the necessary condition is new. Our starting point is completely different from that in \cite{AlekseevXu}. Alekseev \& Xu proved the local existence of Dirac generating operators for general Courant algebroids.
By applying it to the case of Manin triples of Lie bialgebroids, they proved the sufficient condition in Theorem~\ref{thm:C}. Then they derived that 
$\bdirac^2$ is a function as a consequence. 
On the other hand, our approach is, in a certain sense, to take the opposite route.
We prove Theorem~\ref{thm:A} by reducing it to Theorem~\ref{thm:B}, which can be verified by a direct argument. Then Theorem~\ref{thm:C} follows immediately.
Indeed the heart of our approach consists in establishing a list of new identities (see Theorem~\ref{Thm:C}), which we believe deserve attention in their own right.

Note that the derived brackets of the Courant algebroid $TM\oplus T^*M$ have played an important role in the rapid development of the generalized complex geometry of Hitchin and Gualtieri \cites{MR1876068,math/0703298}, where many remarkable results have been established. We hope our result will be of some use in this subject.

The paper is organized as follows.

Section 1 gives a succinct account of standard facts about Courant algebroids, Lie bialgebroids and Dirac generating operators, whose purpose is to fix the notations.

The differential operator $\bdirac$ is defined in Section 2 and the main theorems are then stated without proofs.

Section 3 establishes a list of important identities valid in any pair of Lie algebroids $(A,A^*)$, which are subsequently used in Section 4 to prove the statements of Section 2.

Our results are then particularized to a few concrete situations in Section 5. Namely, we discuss the cases of exact Lie bialgebroids, Poisson Nijenhuis Lie algebroids and finally $a+b$ Lie bialgebras.

\paragraph{Acknowledgments} 
The present work was completed while Zhuo Chen was visiting the mathematics department at Penn State with support from its Shapiro fund. Mathieu Sti\'enon is grateful to Peking University for its hospitality. The authors thank Yvette Kosmannn-Schwarzbach, Jim Stasheff, and   Ping Xu for  useful discussions and comments.

\section{Preliminaries}

\subsection{Lie bialgebroids and Courant algebroids}

A Lie algebroid consists of a vector bundle $A\to M$, a bundle map $\anchor:A\to TM$ called anchor and a Lie algebra bracket $\lb{\cdot}{\cdot}$ on the space of sections $\sections{A}$ such that $\anchor$ induces a Lie algebra homomorphism from $\sections{A}$ to $\XX(M)$ and the Leibniz rule
\begin{equation*}\label{1} \lb{X}{fY}=\big(\anchor(X)f\big)Y+f\lb{X}{Y} \end{equation*}
is satisfied for all $f\in\cinf{M}$ and $X,Y\in\sections{A}$.

It is well-known \cite{MR1675117} that a Lie algebroid $(A,\ba{\cdot}{\cdot},\anchor)$ gives rise to a Gerstenhaber algebra $(\sections{\wedge\graded A},\wedge,\lb{\cdot}{\cdot})$,
and a degree~1 derivation $\dee$ of the graded commutative algebra $(\sections{\wedge\graded A^*},\wedge)$ such that $\dee^2=0$. Here the (Lie algebroid) differential $\dee$ is given by
\begin{multline*}
\label{m2}
(\dee\alpha)(X_0,X_1,\cdots,X_n)=\sum_{i=0}^n (-1)^i \anchor(X_i)
 \alpha(X_0,\cdots,\widehat{X_i},\cdots,X_n) \\
+ \sum_{i<j} (-1)^{i+j} \alpha(\ba{X_i}{X_j},X_0,\cdots,\widehat{X_i},\cdots,\widehat{X_j},\cdots,X_n).
\end{multline*}
To each $X\in\sections{A}$ is associated a degree~-1 derivation $\ii{X}$ of the graded commutative algebra $(\sections{\wedge\graded A^*},\wedge)$, given by
\begin{equation*}\label{2} (\ii{X}\alpha)(X_1,\cdots,X_n)
=\alpha(X,X_1,\cdots,X_n) .\end{equation*}
The Lie derivative $\ld{X}$ in the direction of a section $X\in\sections{A}$ is a degree~0 derivation of the graded commutative algebra $(\sections{\wedge\graded A^*},\wedge)$ defined by the relation $\ld{X}=\ii{X}\dee+\dee\ii{X}$.
The same symbol $\ld{X}$ is also used to denote the derivation of 
the Gerstenhaber algebra $(\sections{\wedge\graded A},\wedge,\ba{\cdot}{\cdot})$ induced by the Lie algebroid bracket: $\ld{X}Y=\ba{X}{Y}$ for any $X,Y\in\sections{A}$.

A Courant algebroid consists of a vector bundle $\pi:E\to M$, a non
degenerate pseudo-metric $\ip{\cdot}{\cdot}$ on the fibers of $\pi$,
a bundle map $\rho:~E\to TM$ called anchor and a $\reals$-bilinear
operation $\db{}{}$ on $\sections{E}$ called Dorfman bracket, which,
for all $f\in\cinf{M}$ and $x,y,z\in\sections{E}$ satisfy the
relations
\begin{align}
& \db{x}{(\db{y}{z})}=\db{(\db{x}{y})}{z}+\db{y}{(\db{x}{z})}; \label{g1} \\
& \rho(\db{x}{y})=\lb{\rho(x)}{\rho(y)};  \\
& \db{x}{fy}=\big(\rho(x)f\big)y+f(\db{x}{y});  \\
& \db{x}{y}+\db{y}{x}=2\DD\ip{x}{y}; \label{g4} \\
& \db{\DD f}{x}=0;  \\
& \rho(x)\ip{y}{z}=\ip{\db{x}{y}}{z}+\ip{y}{\db{x}{z}}
,  \end{align}
where $\DD:\cinf{M}\to\sections{E}$ is the $\reals$-linear map defined by
$\ip{\DD f}{x}=\thalf\rho(x)f$.

The symmetric part of the Dorfman bracket is given by \eqref{g4}.
The Courant bracket is defined as the skew-symmetric part
$\cb{x}{y}=\thalf(\db{x}{y}-\db{y}{x})$ of the Dorfman bracket. Thus we have the relation $\db{x}{y}=\cb{x}{y}+\DD\ip{x}{y}$.

The definition of a Courant algebroid can be rephrased using the Courant bracket instead of the Dorfman bracket \cite{math/9910078}.

A Dirac structure is a smooth subbundle $A\to M$ of the Courant algebroid $E$, which is maximal isotropic with respect to the pseudo-metric and whose space of sections is closed under (necessarily both) brackets.
Thus a Dirac structure inherits a canonical Lie algebroid structure.

Let $A\to M$ be a vector bundle. Assume that $A$ and its dual $A^*$ both carry a Lie algebroid structure with anchor maps
$\anchor:A\to TM$ and $\anchors:A^*\to TM$, brackets on sections
$\sections{A}\otimes_{\reals}\sections{A}\to\sections{A}:u\otimes v\mapsto\ba{u}{v}$ and
$\sections{A^*}\otimes_{\reals}\sections{A^*}\to\sections{A^*}:\theta\otimes \eta\mapsto\bas{\theta}{\eta}$, and differentials
$\dee:\sections{\wedge^{\bullet}A^*}\to\sections{\wedge^{\bullet+1}A^*}$
and $\dees:\sections{\wedge^{\bullet}A}\to\sections{\wedge^{\bullet+1}A}$.

This pair of Lie algebroids $(A,A^*)$ is a Lie bialgebroid (or Manin
triple) \cites{MR1362125,MR1746902,MR1262213} if $\dees$ is a
derivation of the Gerstenhaber algebra $(\sections{\wedge\graded
A},\wedge,\ba{\cdot}{\cdot})$ or, equivalently, if $\dee$ is a
derivation of the Gerstenhaber algebra $(\sections{\wedge\graded
A^*},\wedge,\bas{\cdot}{\cdot})$. Since the bracket
$\bas{\cdot}{\cdot}$ (resp. $\ba{\cdot}{\cdot}$) can be recovered
from the derivation $\dees$ (resp. $\dee$), one is led to the
following alternative definition. A Lie bialgebroid is a pair
$(A,\dees)$ consisting of a Lie algebroid
$(A,\ba{\cdot}{\cdot},\anchor)$ and a degree~1 derivation $\dees$ of
the Gerstenhaber algebra $(\sections{\wedge\graded
A},\wedge,\ba{\cdot}{\cdot})$ such that $\dees^2=0$.

The link between Courant and Lie bialgebroids is given by the following

\begin{thm}[\cite{MR1472888}]
\label{Thm:double} There is a 1-1 correspondence between Lie
bialgebroids and pairs of transversal Dirac structures in a Courant
algebroid.
\end{thm}

More precisely, let $A\to M$ be a vector bundle such that $A$ and its dual $A^*$ both carry a Lie algebroid structure with anchor maps
$\anchor:A\to TM$ and $\anchors:A^*\to TM$, brackets on sections
$\sections{A}\otimes_{\reals}\sections{A}\to\sections{A}:u\otimes v\mapsto\ba{u}{v}$ and
$\sections{A^*}\otimes_{\reals}\sections{A^*}\to\sections{A^*}:\theta\otimes \eta\mapsto\bas{\theta}{\eta}$, and differentials
$\dee:\sections{\wedge^{\bullet}A^*}\to\sections{\wedge^{\bullet+1}A^*}$
and $\dees:\sections{\wedge^{\bullet}A}\to\sections{\wedge^{\bullet+1}A}$. If the pair $(A,A^*)$ is a Lie bialgebroid,
then the vector bundle $A\oplus A^*\to M$ together with the pseudo-metric
\begin{equation}\label{3} \ip{X_1+\xi_1}{X_2+\xi_2}=\thalf\big(\xi_1(X_2)+\xi_2(X_1)\big) ,\end{equation}
the anchor map $\rho=\anchor+\anchors$ (whose dual is given by $\DD f=\dee f+\dees f$ for $f\in\cinf{M}$) and the Dorfman bracket
\begin{equation}\label{4} \db{(X_1+\xi_1)}{(X_2+\xi_2)}=\big(\ba{X_1}{X_2}+\ld{\xi_1}X_2-\ii{\xi_2}(\dees X_1)\big)+\big(\bas{\xi_1}{\xi_2}+\ld{X_1}\xi_2-\ii{X_2}(\dee\xi_1)\big) \end{equation}
is a Courant algebroid of which $A$ and $A^*$ are transverse Dirac structures.
It is called the double of the Lie bialgebroid $(A,A^*)$.
Here $X_1,X_2$ denote arbitrary sections of $A$ and $\xi_1,\xi_2$ arbitrary sections of $A^*$.

\subsection{Dirac generating operators}

Let $V$ be a vector space of dimension $n$ endowed with a non degenerate symmetric bilinear form  $\ip{\cdot}{\cdot}$. Its Clifford algebra $\clifford{V}$ is defined as the quotient of the tensor algebra $\oplus_{k=0}^n V^{\otimes n}$ by the relations $x\otimes y+y\otimes x=2\ip{x}{y}$ ($x,y\in V$).
It is naturally an associative $\ZZ_2$-graded algebra.
Up to isomorphism, there exists a unique irreducible module $S$ of $\clifford{V}$ called spin representation \cite{MR1636473}. The vectors of $S$ are called spinors.

\begin{ex}\label{translagr}
Let $W$ be a vector space of dimension $r$.
We can endow $V=W\oplus W^*$ with the non degenerate pairing
\begin{equation*}\label{5} \ip{w_1+\omega_1}{w_2+\omega_2}=\thalf\big(\omega_1(w_2)+\omega_2(w_1)\big) ,\end{equation*}
where $w_1,w_2\in W$ and $\omega_1,\omega_2\in W^*$.
The representation of $\clifford{V}$ on $S=\oplus_{k=0}^r \wedge^k W$ defined by
$u\cdot w=u\wedge w$ and $\xi\cdot w=\ii{\xi}w$, where $u\in W$, $\xi\in W^*$ and $w\in S$, is the spin representation.
Note that $S$ is $\ZZ$- and thus also $\ZZ_2$-graded.
\end{ex}

Now let $\pi:E\to M$ be a vector bundle endowed with a non degenerate pseudo-metric  $\ip{\cdot}{\cdot}$ on its fibers and let $\clifford{E}\to M$ be the associated bundle of Clifford algebras.
Assume there exists a smooth vector bundle $S\to M$ whose fiber $S_m$ over a point $m\in M$ is the spin module of the Clifford algebra $\clifford{E}_m$. Assume furthermore that $S$ is $\ZZ_2$-graded: $S=S^0\oplus S^1$.

An operator $O$ on $\sections{S}$ is called even (or of degree~0) if $O(S^i)\subset S^{i}$ and odd (or of degree~1) if $O(S^i)\subset S^{i+1}$. Here $i\in\ZZ_2$.

\begin{ex} If the vector bundle $E$ decomposes as the direct sum $A\oplus A^*$ of two transverse Lagrangian subbundles as in Example~\ref{translagr}, then $S=\wedge A$. The multiplication by a function $f\in\cinf{M}$ is an even operator on $\sections{S}$ while the Clifford action of a section $e\in\sections{E}$ is an odd operator on $\sections{S}$.
\end{ex}

If $O_1$ and $O_2$ are operators of degree~$d_1$ and $d_2$
respectively, then their commutator is the operator
$\lb{O_1}{O_2}=O_1\rond O_2-(-1)^{d_1 d_2}O_2\rond O_1$.

\begin{defn}[\cite{AlekseevXu}]
\label{Def:DiracGTR} A Dirac generating operator for $(E,\ip{~}{~})$
is an odd operator $\dirac$ on $\sections{S}$ satisfying the
following properties:
\begin{enumerate}
\item \label{conda} For all $f\in\cinf{M}$, $\lb{\dirac}{f}\in\sections{E}$. This means that the operator $\lb{\dirac}{f}$ is the Clifford action of some section of $E$.
\item \label{condb} For all $e_1,e_2\in\sections{E}$,
$\lb{\lb{\dirac}{e_1}}{e_2}\in\sections{E}$.
\item \label{condc} The square of $\dirac$ is the multiplication by some function on $M$: that is $\dirac^2\in\cinf{M}$.
\end{enumerate}
\end{defn}

The ``deriving operators'' of \cite{MR2103012} are a related more general notion.

We have the following useful properties:
\begin{align*}
& \lb{\lb{\dirac}{\dirac}}{e}=0, \\
& \lb{\dirac}{\lb{e_1}{e_2}}=\lb{\lb{\dirac}{e_1}}{e_2}
-\lb{e_1}{\lb{\dirac}{e_2}}, \\
& \lb{\lb{\dirac}{f}}{e}=\lb{\lb{\dirac}{e}}{f},
\end{align*}
for all $f\in\cinf{M}$ and $e,e_1,e_2\in\sections{E}$.

\begin{thm}[\cite{AlekseevXu}]
Let $\dirac$ be a Dirac generating operator for a vector bundle $\pi:E\to M$. Then there is a canonical Courant algebroid structure on $E$.
The anchor $\rho:E\to TM$ is defined by $\rho(e)f=2\ip{\lb{\dirac}{f}}{e}=\lb{\lb{\dirac}{f}}{e}$,
while the Dorfman bracket reads
$\db{e_1}{e_2}=\lb{\lb{\dirac}{e_1}}{e_2}$.
\end{thm}

\section{Statement of the main theorem}

We follow the same setup as in \cite{AlekseevXu}.

Let $(A,\ba{\cdot}{\cdot},\anchor)$ and
$(A^*,\bas{\cdot}{\cdot},\anchors)$ be a pair of Lie algebroid structures on a rank-$n$ vector bundle $A$ over a dimension-$m$ manifold $M$ and its dual $A^*$. The line bundle $\wedge^n A^*\otimes\wedge^m T^*M$ is a module over the Lie algebroid $A^*$ \cite{MR1726784}: a section $\alpha\in\sections{A^*}$ acts on $\sections{\wedge^n A^*\otimes\wedge^m T^*M}$ by
\begin{equation}\label{6} \alpha\cdot(\alpha_1\wedge\cdots\wedge\alpha_n\otimes\mu)= \sum_{i=1}^n\big(\alpha_1\wedge\cdots\wedge\bas{\alpha}{\alpha_i}\wedge\cdots\wedge\alpha_n\otimes\mu\big) +\alpha_1\wedge\cdots\wedge\alpha_n\otimes\ld{\anchors(\alpha)}\mu .\end{equation}
If it exists, the square root $\module=(\wedge^n A^*\otimes\wedge^m T^*M)^{\thalf}$ of this line bundle is also a module over $A^*$. One can thus define a differential operator
\begin{equation}\label{7} \bdees:\sections{\wedge^k A\otimes\module}\to\sections{\wedge^{k+1} A\otimes\module} .\end{equation}
Similarly, $\explicit$ is --- provided it exists --- a module over $A$. Hence we obtain a differential operator
\begin{equation}\label{8} \sections{\wedge^k A^*\otimes\explicit}\to
\sections{\wedge^{k+1} A^*\otimes\explicit} .\end{equation}
But the isomorphisms of vector bundles
\begin{equation}\label{9} \wedge^k A^*\isomorphism \wedge^k A^* \otimes \wedge^{n-k}A^*\otimes\wedge^{n-k}A\isomorphism\wedge^{n-k}A\otimes\wedge^n A^* \end{equation}
and
\begin{equation}\label{10} \wedge^n A^*\otimes\explicit\isomorphism
(\wedge^n A^*\otimes\wedge^m T^*M)^{\thalf} \end{equation}
imply that
\begin{multline} \label{m1}
\wedge^k A^*\otimes\explicit
\isomorphism \wedge^{n-k}A\otimes\wedge^n A^*\otimes
\explicit \\ \isomorphism
\wedge^{n-k}A\otimes
(\wedge^n A^*\otimes\wedge^m T^*M)^{\thalf}
.\end{multline}
Therefore, one ends up with a differential operator
\begin{equation}\label{11} \bdel:\sections{\wedge^k A\otimes\module}
\to\sections{\wedge^{k-1} A\otimes\module} .\end{equation}


Our main results are the following theorem and its corollary.

\begin{thm}\label{Thm:A}
The pair of Lie algebroids $(A,A^*)$ is a Lie bialgebroid if, and
only if, $\bdirac^2\in\cinf{M}$, i.e. the square of the operator
$\bdirac=\bdees+\bdel$: $\sections{\wedge A\otimes\module}\to
\sections{\wedge A\otimes\module}$ is the multiplication by some function
$\fsmile\in\cinf{M}$. Moreover
$\bdirac^2_*=\fsmile$, where
$\bdiracs=\bdee+\bdels$ is defined analogously to $\bdirac$ by exchanging the roles of $A$ and $\As$.
\end{thm}

\begin{cor}\label{Cor:B}
The pair of Lie algebroids $(A,A^*)$ is a Lie bialgebroid if, and
only if, $\bdirac=\bdees+\bdel$ is a Dirac generating operator for
the bundle $A\oplus A^*$ endowed with the pseudo-metric
\begin{equation*}\label{21} \ip{X_1+\xi_1}{X_2+\xi_2}=\thalf\big(\xi_1(X_2)
+\xi_2(X_1)\big) ,\end{equation*}
where $X_1,X_2\in\sections{A}$ and $\xi_1,\xi_2\in\sections{A^*}$.
\end{cor}

Assume there exists a volume form $s\in\sections{\wedge^m T^*M}$ and a nowhere vanishing section $\Omega\in\sections{\wedge^n A^*}$ so that $\module$ is the trivial line bundle over $M$. And let $V\in\sections{\wedge^n A}$ be the section dual to $\Omega$: $\duality{\Omega}{V}=1$.
These induce two bundle isomorphisms:
\begin{gather}
\Omega\diese:\wedge^k A\to\wedge^{n-k}A^*:X\mapsto \ii{X}\Omega , \label{g10} \\
V\diese:\wedge^k A^*\to\wedge^{n-k}A:\xi\mapsto \ii{\xi}V ,\label{g11}
\end{gather}
which are essentially inverse to each other:

\begin{gather}\label{diverse}
(V^\sharp\circ\Omega^\sharp)(X)=(-1)^{k(n-1)}X,\qquad\forall X\in
\wedge^k A; \\
\label{residents} (\Omega^\sharp\circ V^\sharp)(\varphi)
=(-1)^{k(n-1)}\varphi,\qquad\forall\varphi\in \wedge^k\As.
\end{gather}

Consider the operator $\del$ dual to $\dee$ with respect to $\Omega\diese$:
\begin{equation}\label{12} \xymatrix{ \sections{\wedge^k A^*} \ar[r]^{V\diese} \ar[d]_{-(-1)^k\dee} &
\sections{\wedge^{n-k} A} \ar[d]^{\del} \\
\sections{\wedge^{k+1} A^*} \ar[r]_{V\diese} & \sections{\wedge^{n-k-1} A}, } \end{equation}
or \begin{equation}\label{13} -V\diese\dee\alpha=(-1)^k\del
V\diese\alpha, \quad \forall\alpha\in\sections{\wedge^k A^*},
\end{equation}
which, due to \eqref{diverse} and \eqref{residents}, can be rewritten as
\begin{equation}\label{13bis}
\Omega\diese \del \beta=(-1)^l \dA \Omega\diese \beta,\quad\forall
\beta\in \sections{\wedge^l A}.
\end{equation}

We also have the operator $\dels$ dual to $\dees$:
\begin{equation*}\label{14} \xymatrix{ \sections{\wedge^{n-k} A} \ar[d]_{ (-1)^k\dees} &
\sections{\wedge^{ k} A^*} \ar[l]_{V\diese} \ar[d]^{\dels} \\
\sections{\wedge^{n-k+1} A} & \sections{\wedge^{ k-1} A^*} \ar[l]^{V\diese}, } \end{equation*}
or \begin{equation*}\label{15}  \dees V\diese\alpha=(-1)^k
V\diese\dels \alpha, \quad \forall\alpha\in\sections{\wedge^{ k}
A^*}.
\end{equation*}

The operator $\del$ is a Batalin-Vilkovisky operator for the Lie
algebroid $A$ \cites{MR1764439,MR837203,MR1675117,MR2182214}: indeed
one has $\del^2=0$ and, for any $u\in\sections{\wedge^k A}$ and
$v\in\sections{\wedge^l A}$,
\begin{equation}\label{16} \ba{u}{v}=(-1)^k \big(\del(u\wedge v)-(\del u)\wedge v- (-1)^k u\wedge(\del v)\big) .\end{equation}
Using \eqref{16}, one can also prove that
\[ \del\ba{u}{v}=\ba{\del u}{v} +(-1)^{k+1}\ba{u}{\del v} .\]
Similar relations hold for $\dels$.

Consider the pair of operators $\Dee=\dees+\del$ on $\sections{\wedge A}$ and $\Dees=\dee+\dels$ on $\sections{\wedge A^*}$. Their squares yield the pair of Laplacian operators
\begin{gather}
\lap=\Dee^2=\dees\del+\del\dees:~~ \sections{\wedge^k A}\to\sections{\wedge^k A}, \label{g12} \\
\laps=\Dees^2=\dee\dels+\dels\dee:~~ \sections{\wedge^k
A^*}\to\sections{\wedge^k A^*}. \label{g13}
\end{gather}

There exists a unique
$X_0\in\sections{A}$ such that
\begin{equation}\label{17} \ld{\theta}(\Omega\otimes s)=(\ld{\theta}\Omega)\otimes s+\Omega\otimes(\ld{\anchors(\theta)} s)=\duality{X_0}{\theta}\Omega\otimes s,\quad\forall \theta\in\sections{A^*} .\end{equation}
Similarly, there exists a unique $\xi_0\in\sections{A^*}$ such that
\begin{equation}\label{18} \ld{u}(s\otimes V)=(\ld{\anchor(u)}s)\otimes V+s \otimes(\ld{u}V)=\duality{\xi_0}{u}s\otimes V,\quad\forall u\in\sections{A} .\end{equation}
These sections $X_0$ and $\xi_0$ are called modular cocycles and
their cohomology classes are called modular classes
\cite{MR1726784}.

A simple computation yields that
\begin{equation*}\label{19} \bdees(a\otimes l)=(\dees a+\thalf X_0\wedge a)\otimes l \end{equation*} and \begin{equation*}\label{20} \bdel(a\otimes l)=(-\del a+\thalf \ii{\xi_0}a)\otimes l ,\end{equation*}
for all $a\in\sections{\wedge A}$ and $l\in\sections{\module}$.
Hence
\[ \bdirac=\bdees+\bdel=\dees-\del+\thalf(X_0\wedge+\ii{\xi_0})
.\]

\begin{prop}\label{Pro:commissoner} {
Let   $(A,A^*)$ be a Lie bialgebroid. Then the function
${\fsmile}=\bdirac^2=\bdirac^2_*$ is determined by any of the
following two equalities.
\begin{enumerate}
\item \label{q} $\ld{X_0}(\Omega\otimes
s)=4\fsmile(\Omega\otimes s)$.
\item \label{r}
$\ld{\xi_0}(s\otimes V)=4\fsmile(s\otimes V)$.
\end{enumerate}}
\end{prop}

The proof of Theorem~\ref{Thm:A} is based on the following theorem.

\begin{thm}\label{Thm:C}
Let $A\to M$ be a vector bundle such that $A$ and $A^*$ are each endowed with a Lie algebroid structure. The following six assertions are equivalent:
\begin{enumerate}
\NEW{compteur}
\item \label{a} $\dees\ba{u}{v}=\ba{\dees u}{v} +(-1)^{k-1}\ba{u}{\dees v},\quad\forall u\in\sections{\wedge^k A}, \forall v\in\sections{\wedge^l
A}$.
\item \label{b} $\dee\bas{\theta}{\eta}=\bas{\dee\theta}{\eta}+(-1)^{k-1}\bas{\theta}{\dee\eta},\quad\forall \theta\in\sections{\wedge^k A^*}, \forall \eta\in\sections{\wedge^l
A^*}$.
\item \label{i} $\lap(u\wedge v)=\lap u\wedge v+u\wedge \lap v,\quad\forall u,v\in\sections{\wedge
A}$.
\item \label{j} $\laps(\theta\wedge\eta)=\laps\theta\wedge\eta
+\theta\wedge\laps\eta,\quad\forall \theta,\eta\in\sections{\wedge
A^*}$.
\item \label{k} $\lap=\thalf(\ld{X_0}+\ld{\xi_0}) :\sections{\wedge A}\to\sections{\wedge
A}$.
\item \label{l}
$\laps=\thalf(\ld{X_0}+\ld{\xi_0}) :\sections{\wedge
A^*}\to\sections{\wedge A^*}$.
\STO{compteur}
\end{enumerate}
and they imply the following six equivalent assertions:
\begin{enumerate}
\RCL{compteur}
\item \label{g} $\laps\duality{\theta}{u}=\duality{\laps\theta}{u}+\duality{\theta}{\lap u},\quad\forall
u\in\sections{A},\;\theta\in\sections{A^*}$.
\item \label{h} $\lap\duality{\theta}{u}=\duality{\laps\theta}{u}+\duality{\theta}{\lap u},\quad\forall
u\in\sections{A},\;\theta\in\sections{A^*}$.
\item \label{c}
For all $u\in\sections{A}$ and $\theta\in\sections{A^*}$, the map $\ld{\db{u}{\theta}}-\lb{\ld{u}}{\ld{\theta}}:\sections{A^*}\to\sections{A^*}$ is $\cinf{M}$-linear and its trace is
$\trace\big(\ld{\db{u}{\theta}}-\lb{\ld{u}}{\ld{\theta}}\big)=2\duality{\dees u}{\dee\theta}$.
\item \label{d} For all $\theta\in\sections{A^*}$ and $u\in\sections{A}$, the map $\ld{\db{\theta}{u}}-\lb{\ld{\theta}}{\ld{u}}:\sections{A}\to\sections{A}$ is $\cinf{M}$-linear and its trace is
$\trace\big(\ld{\db{\theta}{u}}-\lb{\ld{\theta}}{\ld{u}}\big)=2\duality{\dee\theta}{\dees u}$.
\item \label{e} $\lap f=\thalf(\ld{X_0}+\ld{\xi_0}) f$ and $\lap u=\thalf(\ld{X_0}+\ld{\xi_0}) u$, for all $f\in\cinf{M}$ and
$u\in\sections{A}$.
\item \label{f}
$\laps f=\thalf(\ld{X_0}+\ld{\xi_0}) f$ and $\laps
\theta=\thalf(\ld{X_0}+\ld{\xi_0}) \theta$, for all $f\in\cinf{M}$
and $\theta\in\sections{A^*}$. \STO{compteur}
\end{enumerate}
\end{thm}

Here are a few consequences of Theorem~\ref{Thm:C}.

\begin{cor}\label{Cor:blistering}
If the pair $(A,\As)$ is a Lie bialgebroid, we have the following relations:
\begin{enumerate}
\RCL{compteur}
\item \label{m} $\ld{X_0}V=-\ld{\xi_0}V$.
\item \label{n} $\ld{X_0}\Omega=-\ld{\xi_0}\Omega$.
\item \label{o} $\dels \xi_0=\del X_0$.
\item \label{p} $\ld{X_0}s=\ld{\xi_0}s$.
\end{enumerate}
\end{cor}

\begin{cor}\label{Cor:brood}If the pair $(A,\As)$ is a Lie bialgebroid, then
 $\lap$ is a derivation of the Gerstenhaber algebra $\wedge A$:
\begin{gather}
\lap(u\wedge v)=\lap u\wedge v+u\wedge\lap v \label{g14}; \\
\lap \ba{u}{v}=\ba{\lap u}{v}+\ba{u}{\lap v} \label{g15}
\end{gather}
for all $u,v\in \sections{\wedge A}$.
\end{cor}

\begin{ex}\label{Ex:TM1}
There is an important example of the preceding construction in Poisson geometry, which we shall now examine. Let $P$
be a Poisson manifold with Poisson tensor $\pi$. For the more general twisted Poisson manifolds, one may consult \cite{MR2223167}. The cotangent bundle $\TsP$ carries a natural Lie algebroid structure, called the cotangent Lie algebroid of the Poisson manifold $(P,\pi)$ \cite{MR996653}.
Its anchor is the bundle map $\pisharp:~\TsP \lon TP$ and its Lie bracket is given by
\begin{equation*}\label{Eqt:fundraisers}
\pibracket{\alpha,\beta}=\LieDer_{\pisharp(\alpha)}\beta-\LieDer_{\pisharp(\beta)}\alpha
-\dA (\pi(\alpha,\beta))=\LieDer_{\pisharp(\alpha)}\beta
-\inserts_{\pisharp(\beta)}d\alpha.
\end{equation*}
for any two $1$-forms $\alpha$ and $\beta$ on $P$.
Moreover, the pair $(TP,\TsP)$ is a Lie bialgebroid. The differential $\dees$ corresponding to the Lie algebroid structure on $\TsP$ is the operator
\[ d_{\pi}=[\pi,~]:\XX^k(P)\to\XX^{k+1}(P) \]
first introduced by Lichnerowicz to define Poisson cohomology
\cite{MR0501133}.

Koszul \cite{MR837203} and Brylinski \cite{MR950556} defined
the Poisson homology operator
$\partial_\pi:~\OO^k(P)\to\OO^{k-1}(P)$ as
\[ \partial_\pi=[ \inserts_{\pi},d]=\inserts_{\pi}\circ d-d\circ\inserts_{\pi} .\]

This boundary operator is intimately related to the modular class of the
Poisson manifold introduced independently by Weinstein \cite{MR1484598} and by Brylinski \& Zuckerman \cite{MR1665693}.
Let us briefly recall its definition. The modular vector field of $P$ with respect to a volume form $\Omega\in\OO^{\TOP}(P)$ is
the derivation $X_{\Omega}$ of the algebra of functions $\cinf{P}$
characterized by
\begin{equation}\label{tongtong} \LieDer_{\pisharp(df)}\Omega = X_{\Omega}(f)\Omega .\end{equation}
It is proved in \cite{MR1726784} that the modular cocycle $X_0$ of the cotangent Lie algebroid $\TsP$ is equal to $2X_{\Omega}$. Moreover,
it is shown in \cites{MR1484598,MR1764439} that $X_{\Omega}=\del\pi$, where
$\del:\sections{\wedge^k TP}\to
\sections{\wedge^{k-1}TP}$ is the operator which generates the Schouten bracket on $\XX\graded(P)$ and is defined by
\[ \partial u = (-1)^{k} ({\Omega}^\sharp)^{-1}\circ d\circ  {\Omega}^\sharp(u), \qquad\forall u\in \sections{\wedge^k TM} .\]

Similarly, the operator $\dels:\OO^k(P)\to\OO^{k-1}(P)$ defined by
\[ \dels(\beta)=(-1)^{n-k-1} \Omega^\sharp\circ d_{\pi}\circ
(\Omega^\sharp)^{-1}(\beta), \qquad\forall\beta\in\OO^k(P) \]
generates the Lie bracket on $\Omega(P)$.

The operators $\partial_{\pi}$ and $\partial_*$ are related in the following way \cite{MR1675117}:
\[ \dels=\partial_{\pi}+\inserts_{X_\Omega} .\]

The Laplacian operators are thus
\[ \laps= \dee\dels+\dels\dee =d\partial_\pi+\partial_\pi d +
d\inserts_{X_{\Omega}}+\inserts_{X_{\Omega}}d =\LieDer_{X_{\Omega}}=\thalf \LieDer_{2X_{\Omega}} .\]
and
\[ \lap= \dees\del+\del\dees= d_\pi\partial +\partial d_\pi
=\lb{\pi}{\partial(~)}+\partial\lb{\pi}{~}=\lb{\partial\pi}{~}
=\LieDer_{X_{\Omega}}=\thalf\LieDer_{2X_{\Omega}} .\]
Since the modular cocycle $\xi_0$ of the Lie algebroid $TP$ is always zero, the above conclusions are in agreement with \ref{k} and \ref{l} of Theorem~\ref{Thm:C}.

Since here $A=TP$ and $A^*=T^*P$, we have
$\module=\wedge^{\TOP}T^*P$ and it is left to the reader to check that the operators $\bdees$ and $\bdel$ defined in \eqref{7} and \eqref{11} are given by
\begin{gather*}
\bdees(a\otimes\Omega)=([{\pi},a]+X_{\Omega}\wedge a)
\otimes\Omega, \\
\bdel(a\otimes\Omega)=-\partial a\otimes\Omega,
\end{gather*}
for all $a\in\sections{\wedge TP}$.

Exchanging the roles of $A$ and $A^*$, the line bundle $\module$ is now the trivial line bundle over $P$ and the operators $\bdee$ and $\bdels$ are given by
\begin{gather*}
\bdee(a\otimes 1)=\dee a\otimes 1, \\
\bdels(a\otimes 1)=-\partial_\pi a\otimes 1,
\end{gather*}
where $a\in\sections{\wedge T^*P}$.

We will see in Section~\ref{last_section} that the squares of both Dirac generating operators $\bdirac=\bdees+\bdel$ and $\bdiracs=\bdee+\bdels$ are zero.
\end{ex}

It is well known \cite{MR1262213} that, if $\anchor$ and $\anchors$ denote the anchor maps of a Lie bialgebroid $(A,A^*)$, the bundle map
\[ \pi\diese=\anchor\rond(\anchors)^*:T^*M\to TM \]
defines a Poisson structure on the base manifold $M$.

\begin{cor}\label{Cor:extremal}
Let $M$ be an orientable manifold with volume form $s\in\OO^{\TOP}(M)$ and let $(A,A^*)$ be a Lie bialgebroid over $M$ with associated Poisson bivector $\pi$. Then the modular vector field of the Poisson manifold $(M,\pi)$ with respect to $s$ is
\begin{equation}\label{Eqt:alluded} X_s=\thalf\big( \anchors(\xi_0)-\anchor(X_0) \big) .\end{equation}
\end{cor}

\begin{proof}
By definition of $\xi_0$,
\[ \ld{\dees f}(s\otimes V)=\duality{\xi_0}{\dees f}s\otimes V
=(\ld{\xi_0}f)s\otimes V \qquad \forall f\in\cinf{M}.\]
Since \[ \anchor(\dees f)=\anchor\rond(\anchors)^*(\derham f)=\pi\diese(\derham f) \qquad \forall f\in\cinf{M},\]
we get
\[ \ld{\dees f}s=\ld{\anchor(\dees f)}s=\ld{\pi\diese(\derham f)}s=X_s(f)\, s \qquad \forall f\in\cinf{M},\]
where we have used the definition \eqref{tongtong} of the
modular class of a Poisson manifold.
On the other hand, it follows from Theorem~\ref{Thm:C} that
\[ \ld{\dees f}V=\del(\dees f)V=(\lap f)V
=\thalf(\ld{X_0}f+\ld{\xi_0}f)V \qquad \forall f\in\cinf{M} .\]
Hence, we obtain
\[ (\ld{\xi_0}f)s\otimes V=X_s(f)s\otimes V+
\thalf(\ld{X_0}f+\ld{\xi_0}f)s\otimes V ,\]
which completes the proof.
\end{proof}

\section{Technicalities}\label{BigFormulas}

The following objects play a crucial role in the proof of
Theorem~\ref{Thm:C}:
\begin{itemize}
\item the sections $\lap(u\wedge v)=\lap u\wedge v+u\wedge \lap v$, where $u,v\in\sections{\wedge A}$;
\item the sections $\dees\ba{u}{v}=\ba{\dees u}{v}-(-1)^k\ba{u}{\dees v}$, where $u\in\sections{\wedge^k A}$ and $v\in\sections{\wedge^l A}$;
\item the operators $\ld{\db{u}{\theta}}-\lb{\ld{u}}{\ld{\theta}}:\sections{A^*}\to\sections{A^*}$, where $u\in\sections{A}$ and $\theta\in\sections{A^*}$;
\item the functions
$\laps\duality{\theta}{u}-\duality{\laps\theta}{u}-\duality{\theta}{\lap u}$, where $u\in\sections{A}$ and $\theta\in\sections{A^*}$;
\item the operators $\ld{\dees f}+\ld{\dee f}:\sections{\wedge A}\to
\sections{\wedge A}$, where $f\in\cinf{M}$;
\item the operators $\lap-\thalf(\ld{X_0}+\ld{\xi_0}):\sections{\wedge A}\to\sections{\wedge A}$.
\end{itemize}
In this section, we will establish a bunch of key relations between
them. The proof of Theorem~\ref{Thm:C} and
Corollary~\ref{Cor:blistering} is deferred to
Section~\ref{MainProofs}.

Throughout this section, $A$ is a smooth vector bundle of rank $n$ over a smooth manifold $M$ of dimension $m$ such that $A$ and its dual $A^*$ both carry a Lie algebroid structure with anchor maps
$\anchor:A\to TM$ and $\anchors:A^*\to TM$, brackets on sections
$\sections{A}\otimes_{\reals}\sections{A}\to\sections{A}:u\otimes v\mapsto\ba{u}{v}$ and
$\sections{A^*}\otimes_{\reals}\sections{A^*}\to\sections{A^*}:\theta\otimes \eta\mapsto\bas{\theta}{\eta}$, and differentials
$\dee:\sections{\wedge^{\bullet}A^*}\to\sections{\wedge^{\bullet+1}A^*}$
and $\dees:\sections{\wedge^{\bullet}A}\to\sections{\wedge^{\bullet+1}A}$. The vector bundle $A\oplus A^*\to M$ is endowed with the pseudo-metric
\eqref{3}
and the Dorfman bracket
\eqref{4}.

Moreover, we assume there exists
a nowhere vanishing section $\Omega\in\sections{\wedge^n A^*}$. And we let $V\in\sections{\wedge^n A}$ be the section dual to $\Omega$: $\duality{\Omega}{V}=1$.
These induce two bundle isomorphisms $\Omega\diese$ and $V\diese$ as in \eqref{g10} and \eqref{g11}.
The operators $\del$, $\dels$, $\lap$ and $\laps$ are defined as earlier by the relations \eqref{12} to \eqref{g13}.

From \eqref{4}, it follows that
\begin{equation*}\label{101} \db{u}{\theta}= -\inserts_{\theta} \ds u + \LieDer_{u}\theta \qquad \text{and} \qquad
\db{\theta}{u} = \LieDer_{\theta}u -\inserts_u {\dA \theta} ,\end{equation*}
for all $u\in\sections{A}$ and $\theta\in\sections{A^*}$.

From now on, we fix a volume form $s$ of $M$, a nowhere vanishing
section $\Omega\in\sections{\wedge^n A^*}$ and its dual
$V\in\sections{\wedge^n A}$. The modular cocycle of the Lie
algebroid $A^*$ (resp. $A$) is the unique section $X_0\in\Gamma(A)$ (resp. $\xi_0\in\Gamma(A^*)$) satisfying \eqref{17} (resp. \eqref{18}).

\subsection{Some ubiquitous lemmata}

\begin{lem}\label{Lem: chores} For all $u\in\sections{A}$ and $\theta\in\sections{A^*}$, one has:
\begin{align}\label{ethnic}
\LieDer_{u}\Omega &= -(\partial u) \Omega, &
\LieDer_{u}V &= (\partial u) V, \\
\label{exposure}
\LieDer_{\theta}\Omega &= (\partial_*\theta)\Omega, & \LieDer_{\theta}V &= -(\partial_*\theta)V.
\end{align}
\end{lem}

\begin{proof}
By \eqref{diverse},
\[ u=(-1)^{n-1}V^\sharp \Omega^\sharp(u)=V^\sharp(\xi), \]
where $\xi=(-1)^{n-1}\Omega^\sharp(u)\in
\Gamma(\wedge^{n-1}\As)$.
Since \[ \dA \xi=(-1)^{n-1} \dA \inserts_u \Omega=(-1)^{n-1} \LieDer_{u}\Omega ,\]
it follows from the definition of $\del$ that
\[ (\partial u) \Omega=\Omega^\sharp(\partial u)=(-1)^{n}\Omega^\sharp
V^\sharp(\dA \xi)=-\Omega^\sharp
V^\sharp(\LieDer_{u}\Omega)=-\LieDer_{u}\Omega
.\]
And the second equality in (\ref{ethnic}) follows immediately from
\[ 0=\LieDer_u\pairing{\Omega}{V}=\pairing{\LieDer_u \Omega}{V}+\pairing{\Omega}{\LieDer_u V} .\]
Finally, the symmetry in the exchange of $A$ and $\As$ implies
\eqref{exposure}.
\end{proof}

\begin{notation*}
In the sequel, $\ld{v}A\otimes B$ means $(\ld{v}A)\otimes B$ rather than $\ld{v}(A\otimes B)$.
\end{notation*}

\begin{lem}\label{Lem:Pentagon}
For all $u\in\sections{A}$ and $\theta\in\sections{A^*}$, one has:
\begin{gather}
\label{catalyst}
\LieDer_{u}\Omega\otimes V=-\Omega\otimes \LieDer_{u}V, \\ \label{capitalization}
\LieDer_{\theta}\Omega\otimes V=-\Omega\otimes
\LieDer_{\theta}V, \\
\label{quashed}
\LieDer_{u}\Omega\otimes \LieDer_{\theta}V=
\LieDer_{\theta}\Omega\otimes \LieDer_{u}V, \\
\label{duplicate}
\LieDer_{\theta}\LieDer_{u}\Omega\otimes V +
\Omega\otimes \LieDer_{\theta}\LieDer_{u}V =-
2\LieDer_{u}\Omega\otimes \LieDer_{\theta}V.
\end{gather}
\end{lem}

\begin{proof}
Equations \eqref{catalyst}, \eqref{capitalization} and
\eqref{quashed} follow directly from Lemma~\ref{Lem: chores}.
Applying $\LieDer_{\theta}$ to both sides of
Equation~\eqref{catalyst} and making use of \eqref{quashed}, we get
\eqref{duplicate}.
\end{proof}

\begin{lem}\label{Lem:treatises}
For all $\theta\in \Gamma(\As)$ and $u\in\sections{A}$, we have
\begin{equation}\label{Eqt:partialinserts}
\partial
\inserts_{\theta}+\inserts_{\theta}\partial=\inserts_{\dA
\theta} \qquad \text{and} \qquad
\partial_*\inserts_{u}+\inserts_{u}\partial_*=\inserts_{\ds u}.
\end{equation}
\end{lem}

\begin{proof}
Given $x\in\Gamma(\wedge^k A)$, there exists $\omega\in \Gamma(\wedge^{n-k}\As)$ such that $x=V^\sharp(\omega)$.
Applying $\partial$ to both sides of
\[ \inserts_{\theta}x=\theta \contract (\omega\contract V)=
(\omega\wedge \theta)\contract V=V^\sharp (\omega\wedge \theta) ,\]
we obtain
\begin{multline*}
\partial(\inserts_{\theta}x)=(-1)^{n-k} V^\sharp \dA
(\omega\wedge\theta)
=(-1)^{n-k} V^\sharp(\dA \omega\wedge \theta +
(-1)^{n-k}\omega\wedge \dA \theta) \\
= (-1)^{n-k}\theta\contract (\dA \omega\contract V ) + \dA \theta
\contract (V^\sharp \omega)= \dA \theta \contract x -
\theta\contract \partial x
. \qedhere \end{multline*}
\end{proof}

\subsection{Key relations}

\begin{prop}\label{Pro:plagiarism}
For any $x\in \Gamma(\wedge^{k}A)$ and $y\in
\Gamma(\wedge^{l}A)$, one has
\begin{equation*} \label{satisfaction}
\Delta(x\wedge y)-(\Delta x)\wedge y- x\wedge (\Delta y) =(-1)^k(\ds \Abracket{x,y}- \Abracket{\ds x,y}+(-1)^k \Abracket{x,\ds y})
.\end{equation*}
In particular, for all $f,g\in\CIM$ and $u,v\in \Gamma(A)$, one has
\begin{gather}\label{enrollment}
\Delta(fg)-f\Delta g-g\Delta f =-\Abracket{\ds f,g}+\Abracket{f,\ds
g}; \\ \label{interaction}
\Delta(fu)-f\Delta u -(\Delta f) u= \ds\Abracket{f,u} -\Abracket{\ds
f, u }+\Abracket{f,\ds u};
\\ \label{crosscultural} \nonumber
\Delta(u\wedge v)-(\Delta u)\wedge v-u\wedge (\Delta v) =\ds\Abracket{u,v}-\Abracket{\ds u,v}-\Abracket{u,\ds v}.
\end{gather}
\end{prop}

The Laplacian $\laps$ enjoys similar properties:
\begin{gather}
\label{enrollment2}
\Delta_*(fg)-f\Delta_* g-g\Delta_* f
=-\Asbracket{\dA f,g}+\Asbracket{f,\dA g} ,\\
\label{interaction2} \nonumber
\Delta_*(f\theta)-f\Delta_* \theta -\Delta_* f\ \theta= \dA \Asbracket{f,\theta} -\Asbracket{\dA f,\theta}+\Asbracket{f,\dA\theta}
.\end{gather}

\begin{proof}
Using the relations \eqref{16} and \eqref{g12}, we find
\begin{align*}
&\ds \Abracket{x,y}- \Abracket{\ds x,y}+(-1)^k \Abracket{x,\ds y}\\
=& (-1)^k\ds\big(\partial (x\wedge y)-(\partial x)\wedge
y-(-1)^k x\wedge (\partial y)\big) \\
&+(-1)^k\big(\partial (\ds x\wedge y)-(\partial \ds x)\wedge y+(-1)^k \ds x\wedge (\partial y)\big) \\
&+\big(\partial (x\wedge \ds y)-(\partial x)\wedge \ds y-(-1)^kx\wedge
(\partial \ds y)\big) \\
=&(-1)^k\big(\ds\partial (x\wedge y)-(\ds \partial x)\wedge y+ (-1)^k
(\partial x)\wedge \ds y \\ & -(-1)^k \ds x\wedge (\partial y) -x \wedge (\ds \partial y)\big) \\
&+(-1)^k\big(\partial \ds( x\wedge y)-(\partial \ds x)\wedge y+(-1)^k \ds x\wedge (\partial y)\big) \\ & -(\partial x)\wedge \ds y -(-1)^k x\wedge (\partial \ds y) \\
=&(-1)^k (\Delta (x\wedge y)-(\Delta x) \wedge y-x \wedge (\Delta y)).
\qedhere \end{align*}
\end{proof}

\begin{prop}\label{Pro:supplemented}
For all $u\in\Gamma(A)$ and $\theta\in \Gamma(\As)$, we have
\begin{gather}\label{nomination}
\pairing{\ds\Abracket{f,u}-\Abracket{\ds f,u}+\Abracket{f,\ds u}}{\theta}
=-\pairing{(\LieDer_{\ds f}
+\LieDer_{\dA f})u}{\theta}
=([\LieDer_{u},\LieDer_{\theta}]
-\LieDer_{u \Dorfman{\theta}} )f,
\\ \label{nomination2}
\pairing{u}{\dA\Asbracket{f,\theta}
-\Asbracket{\dA f,\theta}
+\Asbracket{f,\dA\theta}} =-\pairing{u}{(\LieDer_{\ds f}
+\LieDer_{\dA f}) \theta}
=([\LieDer_{\theta},\LieDer_{u}]
-\LieDer_{\theta\Dorfman{u}})f.
\end{gather}
\end{prop}

\begin{proof}
On the one hand, we have
\begin{align*}
& ([\LieDer_{u},\LieDer_{\theta}]-\LieDer_{u \Dorfman{ \theta}} )f \\
=& \LieDer_u \pairing{\ds f}{\theta}-\LieDer_{\theta}
\pairing{u}{\dA f}- \pairing{\ds f}{\LieDer_u \theta}+ \pairing{\ds
u}{\theta\wedge \dA f} \\
=& \pairing{\ld{u}(\dees f)}{\theta}-
\ld{\theta}\pairing{u}{\dee f}-\pairing{\dees u}{\dee f\wedge\theta} \\
=& \pairing{\ld{u}(\dees f)}{\theta}
- \anchors(\theta)\pairing{u}{\dee f}
- \anchors(\dee f)\pairing{u}{\theta}
+ \anchors(\theta)\pairing{u}{\dee f}
+ \pairing{u}{\bas{\dee f}{\theta}} \\
=& -\pairing{(\ld{\dees f}+\ld{\dee f})u}{\theta}
.\end{align*}
And on the other hand, since
\[ \pairing{\ii{\dee f}\dees u}{\theta} =
\pairing{\dees u}{(\dee f)\wedge\theta} =
\ld{\dee f}\pairing{u}{\theta} -
\ld{\theta}\pairing{u}{\dee f} -
\pairing{u}{\ba{\dee f}{\theta}} ,\]
we have
\begin{align*}
& \pairing{\dees\ba{f}{u}-\ba{\dees f}{u}+\ba{f}{\dees u}}{\theta} \\
=& -\ld{\theta}\pairing{u}{\dee f}+\pairing{\theta}{\ba{u}{\dees f}}
-\pairing{\ip{\dee f}\dees u}{\theta} \\
=& -\pairing{\theta}{(\ld{\dees f}+\ld{\dee f})u}
.\qedhere \end{align*}
\end{proof}

\begin{prop}\label{Pro:gubernatorial}
\begin{equation*}\label{precipice}
(\LieDer_{u\Dorfman\theta}-[\LieDer_{u},\LieDer_{\theta}]) \Omega\otimes V
=\big(2\pairing{\ds u}{\dA \theta}
-(\pairing{\Delta u}{\theta}
+\pairing{u}{\Delta_*\theta}
-\Delta_*\pairing{u}{\theta})\big)
\Omega\otimes V
.\end{equation*}
\end{prop}

We need the following lemma.

\begin{lem}\label{Lem:chronological}For all $u\in \Gamma(A)$ and $\theta\in \Gamma(\As)$,
\begin{equation}\label{continuation}
\pairing{\Delta u}{\theta}\Omega\otimes V
=\pairing{\ds u}{\dA \theta}\Omega\otimes V +\LieDer_{\inserts_\theta\ds u}\Omega\otimes
V-\LieDer_{\theta}\LieDer_{u}\Omega\otimes V
-\LieDer_{u}\Omega\otimes \LieDer_{\theta}V.
\end{equation}
\end{lem}

\begin{proof}
By Lemma~\ref{Lem:treatises}, one has
\[ \pairing{\partial \ds u}{\theta}= \pairing{\ds u}{\dA \theta} -\partial (\inserts_{\theta} \ds u)
.\]
Therefore
\[ \pairing{\Delta u}{\theta}
=\pairing{\partial \ds u+ \ds \partial u}{\theta}
=\pairing{\ds u}{\dA \theta} -\partial (\inserts_{\theta} \ds u)
+\LieDer_{\theta}(\partial u)
.\]
On the other hand, applying $\LieDer_{\theta}$ to both sides of
$(\partial u) \Omega\otimes V=\Omega\otimes \LieDer_{u} V$ (see Equation~\eqref{ethnic}),
we obtain
\[
\LieDer_{\theta}(\partial u)\Omega \otimes V +(\partial u)\LieDer_{\theta}
\Omega\otimes V +(\partial u) \Omega \otimes \LieDer_{\theta}V
= \LieDer_{\theta}\Omega\otimes \LieDer_{u}V+\Omega\otimes
\LieDer_{\theta}\LieDer_{u}V
.\]
Applying Equation~\eqref{ethnic} again, and then Equation~\eqref{duplicate}, we get
\[ \LieDer_{\theta}(\partial u)\Omega\otimes V
=\Omega\otimes \LieDer_{\theta}\LieDer_{u}V
+\LieDer_{u}\Omega\otimes \LieDer_{\theta}V
= -\LieDer_{u}\Omega\otimes \LieDer_{\theta}V
-\LieDer_{\theta}\LieDer_{u}\Omega\otimes V
.\]
So we have
\begin{multline*}
\pairing{\Delta u}{\theta}\Omega\otimes V=\pairing{\ds u}{\dA
\theta}\Omega\otimes V -\partial (\inserts_{\theta} \ds
u)\Omega\otimes V + \LieDer_{\theta}(\partial
u)\Omega\otimes V \\
= \pairing{\ds u}{\dA \theta}\Omega\otimes V + \LieDer_{
\inserts_{\theta} \ds u}\Omega\otimes V -  \LieDer_{u}\Omega\otimes
\LieDer_{\theta}V -\LieDer_{\theta}\LieDer_{u}\Omega\otimes V
. \qedhere \end{multline*}
\end{proof}

\begin{proof}[Proof of Proposition~\ref{Pro:gubernatorial}]
By the symmetry in the exchange of $A$ and $\As$ in
Lemma~\ref{Lem:chronological}, we get
\begin{equation}\label{continuation2}
\pairing{u}{\Delta_*\theta}\Omega\otimes V =\pairing{\ds u}{\dA\theta}\Omega\otimes V
+\Omega\otimes\LieDer_{\inserts_u\dA\theta} V
-\Omega\otimes\LieDer_u\LieDer_{\theta} V -\LieDer_{u}\Omega\otimes\LieDer_{\theta}V.
\end{equation}
Adding \eqref{continuation} to \eqref{continuation2} and making use of \eqref{duplicate} to simplify,
we obtain:
\begin{eqnarray*}
&& (\pairing{\Delta u}{\theta}
+\pairing{u}{\Delta_*\theta})\Omega\otimes V \\
&=& 2\pairing{\ds u}{\dA\theta}\Omega\otimes V
+\Omega\otimes\LieDer_{\inserts_u\dA\theta} V+ \LieDer_{\inserts_\theta\ds u}\Omega\otimes V
+ \Omega\otimes\LieDer_{\theta}\LieDer_{u} V
- \Omega\otimes\LieDer_{u}\LieDer_{\theta} V \\
&=& 2\pairing{\ds u}{\dA\theta}\Omega\otimes V
+\Omega\otimes\LieDer_{\inserts_u\dA\theta} V-\Omega\otimes\LieDer_{\inserts_\theta\ds u} V + \Omega\otimes [\LieDer_{\theta},\LieDer_{u}] V.
\end{eqnarray*}
We also notice that, by \eqref{exposure},
\[ (\Delta_*\pairing{u}{\theta})\Omega\otimes V
= (\partial_*\dA\inserts_u\theta)\Omega\otimes V
= -\Omega\otimes\LieDer_{\dA\inserts_u \theta} V .\]
So, the subtraction of the last two equations above yields
\begin{eqnarray*}
&&(\pairing{\Delta u}{\theta}+\pairing{u}{\Delta_*
\theta}-\Delta_*\pairing{u}{\theta})\Omega\otimes V
\\  &=&
2\pairing{\ds u}{\dA \theta}\Omega\otimes V + \Omega\otimes
(\LieDer_{u\Dorfman \theta }-[\LieDer_{u},\LieDer_{\theta}])  V
\\  &=&
2\pairing{\ds u}{\dA \theta}\Omega\otimes V - (\LieDer_{u\Dorfman
\theta }-[\LieDer_{u},\LieDer_{\theta}])\Omega\otimes V \quad \text{(by \eqref{catalyst}, \eqref{capitalization} and \eqref{quashed})},
\end{eqnarray*}
as required.
\end{proof}

\begin{prop}\label{Pro:steward} If
\begin{equation}\label{turmoil} \Delta_*\pairing{u}{\theta}=\pairing{\Delta
u}{\theta}+\pairing{u}{\Delta_* \theta},
\end{equation}
holds for all $u\in \Gamma(A)$ and $\theta\in \Gamma(\As)$, then for
any $f\in \CIM$,
\[ \LieDer_{\ds f}+\LieDer_{\dA f}=0
\qquad \text{as a map }
\Gamma(\wedge A)\to\Gamma(\wedge A).
\] \end{prop}

Together with \eqref{ethnic} and \eqref{exposure}, this implies that
\[ (\lap f)\Omega=(\dels\dee f)\Omega=\ld{\dee f}\Omega
=-\ld{\dees f}\Omega=(\del\dees f)\Omega=(\laps f)\Omega ,\]
i.e. $\lap f=\laps f$ for all $f\in\cinf{M}$.
Therefore, \eqref{turmoil} is equivalent to
\begin{equation*}\label{turmoil2} \Delta\pairing{u}{\theta}=\pairing{\Delta
u}{\theta}+\pairing{u}{\Delta_* \theta}.
\end{equation*}

\begin{proof}[Proof of Proposition~\ref{Pro:steward}]
By \eqref{enrollment2}, we have
\begin{align*}
\Delta_*\pairing{u}{f\theta}
=\Delta_*(f\pairing{u}{\theta})
=(\Delta_* f)\pairing{u}{\theta}
+ f\Delta_*\pairing{u}{\theta}
-\Asbracket{{\dA f},{\pairing{u}{\theta}}}
+\Asbracket{{f},{\dA\pairing{u}{\theta}}} \\
=(\Delta_* f)\pairing{u}{\theta}
+ f\Delta_*\pairing{u}{\theta}
-(\LieDer_{\ds f}+\LieDer_{\dA f})
\pairing{u}{\theta}
.\end{align*}
On the other hand, by \eqref{interaction2} and Proposition~\ref{Pro:supplemented}, we have
\begin{eqnarray*}
&& \pairing{\Delta u}{f\theta}
+\pairing{u}{\Delta_*(f\theta)} \\
&=& (\Delta_* f)\pairing{u}{\theta}
+ f\pairing{u}{\Delta_*\theta}
+ f\pairing{\Delta u}{\theta}
+ \pairing{u}{\dA\Asbracket{f,\theta}
-\Asbracket{\dA f,\theta}
+ \Asbracket{f,\dA\theta}} \\
&=& (\Delta_* f)\pairing{u}{\theta}
+ f\pairing{u}{\Delta_*\theta}
+ f\pairing{\Delta u}{\theta}
- \pairing{u}{(\LieDer_{\ds f}+\LieDer_{\dA f}) \theta}
.\end{eqnarray*}
So, if \eqref{turmoil} holds for arbitrary
$u$ and $\theta$, then
\[ (\LieDer_{\ds f}+\LieDer_{\dA f})
\pairing{u}{\theta}=
\pairing{u}{(\LieDer_{\ds f}+\LieDer_{\dA f})  \theta}
.\]
The conclusion follows immediately.
\end{proof}

\begin{prop}\label{Pro:shrugged}
For any $f\in\CIM$,
\begin{multline}\label{Relaxed}
(2\Delta f-(\LieDer_{X_0}+\LieDer_{\xi_0})f)
\Omega\otimes s\otimes V \\
= \Omega\otimes s\otimes (\LieDer_{\dA f} +\LieDer_{\ds f})V-\Omega\otimes
(\LieDer_{\dA f}+\LieDer_{\ds f})s\otimes V.
\end{multline}
\end{prop}

\begin{proof} The proof is a direct calculation:
\begin{eqnarray*}
&& (2\Delta f-(\LieDer_{X_0}+\LieDer_{\xi_0})f)
\Omega\otimes s\otimes V \\
&=& (2\partial \ds f-\pairing{X_0}{\dA f}
-\pairing{\ds f}{\xi_0}) \Omega\otimes s\otimes V \\
&=& 2\Omega\otimes s\otimes\LieDer_{\ds f}V -\LieDer_{\dA f}(\Omega\otimes s)\otimes V
- \Omega\otimes \LieDer_{\ds f}(s\otimes V) \\
&=& \Omega\otimes s\otimes\LieDer_{\ds f}V -\Omega\otimes(\LieDer_{\dA f}+\LieDer_{\ds f})
s\otimes V -\LieDer_{\dA f}\Omega\otimes s\otimes V
.\end{eqnarray*}
The result follows from \eqref{capitalization}.
\end{proof}

\begin{prop}\label{Pro:prelude}
\begin{multline}\label{postface}
\pairing{2\Delta u
-(\LieDer_{X_0}+\LieDer_{\xi_0})u}{\theta}\Omega\otimes s\otimes V = 2\pairing{\ds u}{\dA \theta}\Omega\otimes s\otimes V \\
+([\LieDer_{u},\LieDer_{\theta}]-\LieDer_{u\Dorfman
\theta})\Omega\otimes s \otimes V + \Omega \otimes
([\LieDer_{u},\LieDer_{\theta}]-\LieDer_{u\Dorfman \theta})s \otimes V.
\end{multline}
\end{prop}

We will need two lemmas.

\begin{lem}\label{Lem:intermediate}
For all $u\in\Gamma(A)$ and $\theta\in\Gamma(\As)$,
\begin{multline}\label{grants}
\pairing{u}{\xi_0}\Omega\otimes \LieDer_{\theta}(s\otimes V)
-\pairing{X_0}{\theta}\LieDer_{u}
(\Omega\otimes s)\otimes
V \\ =2(\Omega\otimes \LieDer_{\theta}s\otimes
\LieDer_{u}V- \LieDer_{\theta}\Omega\otimes \LieDer_{u}s\otimes V)
.\end{multline}
\end{lem}

\begin{proof}
We notice that
\begin{eqnarray*}
&&\pairing{u}{\xi_0}\Omega\otimes \LieDer_{\theta}(s\otimes V) \\
&=& \pairing{u}{\xi_0}(\Omega\otimes \LieDer_{\theta}s\otimes V+
\Omega\otimes  s\otimes \LieDer_{\theta}V) \\
&=& \pairing{u}{\xi_0}(\Omega\otimes \LieDer_{\theta}s\otimes V-
\LieDer_{\theta}\Omega\otimes s\otimes V) \\
&=& \pairing{u}{\xi_0}(\LieDer_{\theta}
(\Omega\otimes s)\otimes V-
2\LieDer_{\theta}\Omega\otimes s\otimes V) \\
&=& \pairing{u}{\xi_0}\pairing{X_0}{\theta}
\Omega\otimes s\otimes V-
2\LieDer_{\theta}\Omega\otimes\LieDer_{u}
(s\otimes V)
\quad \text{(by definition of $\xi_0$)} \\
&=& \pairing{u}{\xi_0}\pairing{X_0}{\theta}
\Omega\otimes s\otimes V-
2\LieDer_{\theta}\Omega\otimes\LieDer_{u}
s\otimes V -2
\LieDer_{\theta}\Omega \otimes  s\otimes \LieDer_{u}V.
\end{eqnarray*}
For the same reasons we have
\begin{multline*}
\pairing{X_0}{\theta}\LieDer_{u}(\Omega\otimes  s)\otimes V
= \pairing{u}{\xi_0}\pairing{X_0}{\theta}
\Omega\otimes s\otimes V -2\Omega\otimes \LieDer_{\theta}s\otimes\LieDer_{u}V \\
-2\LieDer_{\theta}\Omega\otimes s\otimes \LieDer_{u}V
.\end{multline*}
The subtraction of these two equalities yields the result.
\end{proof}

\begin{lem}\label{ChenZhuo}
For all $u\in\Gamma(A)$ and $\theta\in\Gamma(\As)$,
\begin{multline}\label{simultaneous}
\pairing{(\LieDer_{X_0}+\LieDer_{\xi_0})u}{\theta}
\Omega\otimes s\otimes V \\
= \LieDer_{\inserts_\theta\ds u}\Omega\otimes s\otimes V +\LieDer_{\LieDer_{u}\theta}\Omega\otimes
s\otimes V -\LieDer_{u}\LieDer_{\theta}\Omega\otimes
s\otimes V -\LieDer_{\theta}\LieDer_{u}\Omega\otimes s\otimes V \\ - 2\LieDer_{u}\Omega\otimes s\otimes \LieDer_{\theta}V +\Omega\otimes (\LieDer_{u\Dorfman{\theta}}
-[\LieDer_{u},\LieDer_{\theta}])s\otimes V.
\end{multline}
\end{lem}

\begin{proof}
Applying $\LieDer_{u}$ to both sides of
$\pairing{X_0}{\theta}\Omega\otimes s=
\LieDer_{\theta}(\Omega\otimes s)$, we get
\begin{align*}
\pairing{\Abracket{u,X_0}}{\theta}\Omega\otimes s
=& \LieDer_{u}\LieDer_{\theta}(\Omega\otimes s)
-\pairing{X_0}{\theta}\LieDer_{u}(\Omega\otimes s)
-\pairing{X_0}{\LieDer_{u}\theta} \Omega\otimes s \\
=& \LieDer_u\LieDer_{\theta}\Omega\otimes
s+\LieDer_{\theta}\Omega\otimes \LieDer_u s + \LieDer_u\Omega\otimes
\LieDer_{\theta} s+\Omega\otimes \LieDer_u\LieDer_{\theta}s \\
& -\pairing{X_0}{\theta}\LieDer_u\Omega\otimes s -
\pairing{X_0}{\theta}\Omega\otimes \LieDer_us-
\LieDer_{\LieDer_u\theta}(\Omega\otimes s)
.\end{align*}
By the symmetry in the exchange of $A$ and $\As$, we get
\begin{align*}
\pairing{u}{\Asbracket{\theta,\xi_0}}s\otimes V
=& \LieDer_{\theta}\LieDer_u s\otimes V + \LieDer_{\theta}s\otimes
\LieDer_{u} V + \LieDer_u s \otimes
\LieDer_{\theta} V + s\otimes \LieDer_{\theta}\LieDer_u V \\
&\quad -\pairing{\xi_0}{u} s\otimes \LieDer_{\theta}V -
\pairing{\xi_0}{u}\LieDer_{\theta}s\otimes V-
\LieDer_{\LieDer_\theta u}(s\otimes V)
.\end{align*}
Therefore, one has
\begin{eqnarray*}
&& \pairing{(\LieDer_{X_0}+\LieDer_{\xi_0})u}{\theta}
\Omega\otimes s\otimes V \\
&=& (-\pairing{\Abracket{u,X_0}}{\theta}
+\pairing{\xi_0}{\ds\pairing{u}{\theta}}
+\pairing{u}{\Asbracket{\theta,\xi_0}})
\Omega\otimes s\otimes V \\
&=& \Omega\otimes \LieDer_{\ds\pairing{u}{\theta}} (s\otimes V) -\LieDer_{u}\LieDer_{\theta}
\Omega\otimes s\otimes V + \Omega\otimes s\otimes
\LieDer_{\theta}\LieDer_{u}V \\
&& + \Omega\otimes
([\LieDer_{\theta},\LieDer_u]+\LieDer_{\LieDer_u \theta}
-\LieDer_{\LieDer_\theta u})s\otimes V -\LieDer_{\theta}\Omega\otimes \LieDer_{u}s\otimes V
+\Omega\otimes \LieDer_{u}s\otimes \LieDer_{\theta}V \\
&& -\LieDer_{u}\Omega\otimes\LieDer_{\theta}s\otimes V +\Omega\otimes\LieDer_{\theta}s\otimes\LieDer_{u}V -\pairing{u}{\xi_0}\Omega\otimes\LieDer_{\theta}
(s\otimes V) \\ && +\pairing{X_0}{\theta}\LieDer_{u}
(\Omega\otimes s)\otimes V +
\LieDer_{\LieDer_u\theta}\Omega\otimes s\otimes V -\Omega\otimes s\otimes\LieDer_{\LieDer_\theta u}V \\
&=& - \LieDer_{u}\LieDer_{\theta}\Omega\otimes s\otimes V + \Omega\otimes s\otimes\LieDer_{\theta}\LieDer_{u}V  -\pairing{u}{\xi_0}\Omega\otimes\LieDer_{\theta}
(s\otimes V) \\
&& +\pairing{X_0}{\theta}\LieDer_{u}
(\Omega\otimes s)\otimes V +
2(\Omega\otimes \LieDer_{\theta}s\otimes \LieDer_{u}V-
\LieDer_{\theta}\Omega\otimes \LieDer_{u}s\otimes V) \\
&& +\Omega\otimes ([\LieDer_{\theta},\LieDer_u]+\LieDer_{u\Dorfman
\theta})s\otimes V+ \LieDer_{\LieDer_u\theta}\Omega\otimes s\otimes
V+\LieDer_{\inserts_\theta \ds u} \Omega\otimes s\otimes V
.\end{eqnarray*}
Then \eqref{simultaneous} follows immediately from \eqref{duplicate} and \eqref{grants}.
\end{proof}

\begin{proof}[Proof of Proposition~\ref{Pro:prelude}]
This is a direct consequence of \eqref{continuation} and
\eqref{simultaneous}.
\end{proof}

\section{Proof of the main theorem}
\label{MainProofs}

\begin{lem}\label{Lem:prejudice}
If $(A,\As)$ is a Lie bialgebroid, then
\begin{equation*}\label{prejudice}
\pairing{X}{(\LieDer_{u\Dorfman \theta}-[\LieDer_u,\LieDer_{\theta}]
)\xi}= \pairing{\inserts_{\xi} \ds u}{\inserts_X \dA\theta},
\end{equation*}
for all $u,X\in \Gamma(A)$ and $\xi,\theta\in \Gamma(\As)$.
\end{lem}

\begin{proof}
From the definition of the Dorfman bracket, it follows that
\begin{equation*}
\pairing{X}{\LieDer_{W}\xi}
=2\ppairing{X,\LieDer_W\xi}
=2\ppairing{X,W\Dorfman\xi}
,\end{equation*}
for any $W\in\sections{A\oplus A^*}$.
Hence we get
\begin{equation}\label{immersed}
\pairing{X}{\LieDer_u\LieDer_{\theta}\xi}
=2\ppairing{X,u\Dorfman(\theta\Dorfman \xi)}
\end{equation}
and
\begin{equation}\label{immersed2}
\begin{aligned}
\pairing{X}{\LieDer_{\theta}\LieDer_u\xi}
=& 2\ppairing{X,\theta\Dorfman(\LieDer_u\xi)} \\
=& 2(\LieDer_{\theta}\ppairing{X,\LieDer_u\xi}
-\ppairing{\theta\Dorfman X,\LieDer_u\xi}) \\
=& 2(\LieDer_{\theta}\ppairing{X,u\Dorfman\xi}
-\ppairing{\theta\Dorfman X,u\Dorfman\xi}
-\ppairing{\theta\Dorfman X,\inserts_{\xi}\ds u}) \\
=& 2\ppairing{X,\theta\Dorfman(u\Dorfman\xi)}+
\pairing{\inserts_{\xi}\ds u}{\inserts_X\dA\theta}.
\end{aligned}
\end{equation}
Subtracting \eqref{immersed} from \eqref{immersed2} yields the result.
\end{proof}

\begin{proof}[Proof of Theorem~\ref{Thm:C}]
\textsl{\underline{Step 1}: We start by proving that \ref{a} implies \ref{c}.}

Because of \eqref{nomination},
for all $f\in\cinf{M}$ and $\alpha\in\sections{A^*}$, one has
\[ \big(\ld{\db{u}{\theta}}-\lb{\ld{u}}{\ld{\theta}}\big)(f\alpha)
= \pairing{\dees\ba{f}{u}-\ba{\dees f}{u}+\ba{f}{\dees u}}{\theta}\alpha + f \big(\ld{\db{u}{\theta}}
-\lb{\ld{u}}{\ld{\theta}}\big)(\alpha) .\]
Therefore, \ref{a} implies that $\ld{\db{u}{\theta}}
-\lb{\ld{u}}{\ld{\theta}}$ is a $\cinf{M}$-linear endomorphism of $\sections{\wedge A^*}$.

Let $\set{X_1,\cdots, X_n}$ be a local basis of $\Gamma(A)$ and let $\set{\xi^1,\cdots,\xi^n}$ be the dual basis of $\Gamma(\As)$. Then it follows from Lemma~\ref{Lem:prejudice} that
\begin{equation*}
\trace(\LieDer_{u\Dorfman\theta}-[\LieDer_u,\LieDer_{\theta}])
= \sum_{i=1}^n \pairing{X_i}
{(\LieDer_{u\Dorfman\theta}-[\LieDer_u,\LieDer_{\theta}])\xi^i}
=\pairing{\inserts_{\xi^i}\ds u}
{\inserts_{X_i}\dA\theta}
= 2\pairing{\ds u}{\dA\theta}
.\end{equation*}

\textsl{\underline{Step 2}: We prove that all assertions of the second group are equivalent.}
\begin{itemize}
\item \fbox{\ref{g} $\Leftrightarrow$ \ref{h}}
First note that the equivalence of \ref{g} and \ref{h} was already
shown to be a consequence of Proposition~\ref{Pro:steward}.
\item \fbox{\ref{g} $\Rightarrow$ \ref{c}}
Assuming \ref{g} holds, the  Propositions~\ref{Pro:steward}
and~\ref{Pro:supplemented} imply that
\begin{equation*} (\LieDer_{u\Dorfman
\theta}-[\LieDer_u,\LieDer_{\theta}] )f=0,\quad\forall f \in \CIM.
\end{equation*}
Thus the map $\LieDer_{u\Dorfman\theta}
-[\LieDer_u,\LieDer_{\theta}] :\sections{A^*}\to\sections{A^*}$ is
$\cinf{M}$-linear and it makes sense to speak of its trace, which is
equal to $2\pairing{\ds u}{\dA\theta}$ by
Proposition~\ref{Pro:gubernatorial}. This proves \ref{c}.
\item \fbox{\ref{c} $\Rightarrow$ \ref{e}}
First, the assumption that
$\LieDer_{u\Dorfman\theta}-[\LieDer_u,\LieDer_{\theta}]$
is $\CIM$-linear implies
that
\begin{equation}\label{passion} (\LieDer_{u\Dorfman
\theta}-[\LieDer_u,\LieDer_{\theta}] )f=0,\quad\forall f \in \CIM.
\end{equation}
Then, by \eqref{nomination}, we have
\begin{equation}\label{cry} \LieDer_{\ds f}+\LieDer_{\dA f}=0,
\quad\text{as a map }
\Gamma(\wedge A)\to\Gamma(\wedge A)
\end{equation}
and the first term of the r.h.s. of
\eqref{Relaxed} vanishes.
Moreover, the second term of the r.h.s. of \eqref{Relaxed} is also zero. Indeed, the expression $s=g\,dx_1\wedge\cdots\wedge dx_n$ of $s$ in local coordinates leads to
\begin{multline*} (\ld{\dee f}+\ld{\dees f})s
= (\ld{\anchors(\dee f)+\anchor(\dees f)} g)
\,dx_1\wedge\cdots\wedge dx_n \\
+g\sum_{i=1}^n dx_1\wedge\cdots\wedge
d(\ld{\anchors(\dee f)+\anchor(\dees f)} x_i)
\wedge\cdots\wedge dx_n ,\end{multline*}
whose r.h.s. must vanish
since \eqref{cry} implies that
\[ (\ld{\dee f}+\ld{\dees f}) g =
\ld{\anchors(\dee f)+\anchor(\dees f)} g =0, \qquad \forall
g\in\cinf{M} .\] Therefore, the l.h.s. of \eqref{Relaxed} is zero,
i.e. \[ \Delta f= \thalf (\LieDer_{X_0} + \LieDer_{\xi_0})f, \qquad
\forall f\in\cinf{M} .\] On the other hand, it follows from
Proposition~\ref{Pro:prelude} that
\[ \Delta u= \thalf (\LieDer_{X_0} + \LieDer_{\xi_0})u,\quad \forall u\in \Gamma(A) .\]
Indeed, the first two terms of the r.h.s. of \eqref{postface} cancel out since
\begin{equation*}
\trace(\LieDer_{u\Dorfman\theta}-[\LieDer_u,\LieDer_{\theta}])
= 2\pairing{\ds u}{\dA\theta}
,\end{equation*}
and the last term of the r.h.s. of \eqref{postface} is zero as a consequence of \eqref{passion}.
This proves \ref{e}.
The same argument shows that \ref{d} implies \ref{f}.
\item \fbox{\ref{e} $\Rightarrow$ \ref{g}}
Finally, it is clear that
\ref{e} (resp. \ref{f}) implies
\ref{g} (resp. \ref{h}).
\end{itemize}

\textsl{\underline{Step 3}: We prove that all assertions of the first group are equivalent.}

The equivalence of \ref{a} and \ref{b} is a well known fact \cites{MR1262213,MR1472888}.
Hence, it follows directly from Proposition~\ref{Pro:plagiarism} that \ref{a}, \ref{b}, \ref{i} and \ref{j} are all equivalent.
Finally, since $\text{\ref{i}}\Leftrightarrow
\text{\ref{a}} \Rightarrow \text{\ref{c}} \Rightarrow \text{\ref{e}}$,
we get that \ref{i} implies \ref{k}.
The converse implication is trivial.
A similar argument shows that \ref{j} is equivalent to \ref{l}.
\end{proof}

\begin{proof}[Proof of Corollary~\ref{Cor:blistering}]
Taking $\alpha=1\in\cinf{M}$ in \eqref{13}, we clearly see that
$\del V=0$. Moreover, $\dees V=0$ since $V$ is of top degree. Thus
\[ \lap V=\del\dees V+\dees\del V=0 ,\] and by \ref{k} of
Theorem~\ref{Thm:C}, we get $(\ld{X_0}+\ld{\xi_0})V=0$. This proves
\ref{m}. A similar argument yields \ref{n}. Now \ref{o} follows from
\ref{m} and the equalities of Lemma~\ref{Lem: chores}:
\[ (\del X_0)V=\ld{X_0}V=-\ld{\xi_0}V=(\dels\xi_0)V. \]
To prove \ref{p}, we observe that
\begin{eqnarray*}
\duality{X_0}{\xi_0}\Omega\otimes s\otimes V
&=& \Omega\otimes \ld{X_0}(s\otimes V) \\
&=& \Omega\otimes \ld{X_0}s\otimes V + \Omega\otimes s\otimes
\ld{X_0} V \\
&=& \Omega\otimes \ld{X_0}s\otimes V -\ld{X_0}
\Omega\otimes s\otimes V \\
&=& \Omega\otimes\ld{X_0}s\otimes V+\ld{\xi_0}\Omega\otimes V \\
&=& \Omega\otimes(\ld{X_0}-\ld{\xi_0})s\otimes V+\ld{\xi_0}(\Omega\otimes s)\otimes V \\
&=& \Omega\otimes (\ld{X_0}-\ld{\xi_0})s\otimes
V+\duality{X_0}{\xi_0}\Omega\otimes s\otimes V
\end{eqnarray*}
forces $(\ld{X_0}-\ld{\xi_0})s$ to be zero.
\end{proof}

\begin{proof}[Proof of Corollary~\ref{Cor:brood}]
Equation~\eqref{g14} is obvious since \ref{a} is equivalent to
\ref{i} in Theorem~\ref{Thm:C}. It is not hard to establish the
following identities:
\begin{gather*}
\LieDer_{\xi}\Abracket{u,v}-\Abracket{\LieDer_{\xi}u,v}
-\Abracket{u,\LieDer_{\xi}v} = \inserts_{(\dA \xi)^\sharp v} \ds u
-\LieDer_{(\dA \xi)^\sharp u}v , \\
\LieDer_{\xi}\Abracket{u,f}-\Abracket{\LieDer_{\xi}u,f}
-\Abracket{u,\LieDer_{\xi}f}= -\duality{u}{\LieDer_{\ds f+\dA
f}\xi}-\LieDer_{(\dA \xi)^\sharp u}f ,
\end{gather*}
for all $u,v\in\sections{A}$, $\xi\in\sections{A^*}$ and
$f\in\cinf{M}$. Since $\dA\xi_0=0$, \eqref{g15} follows immediately
from the implication \ref{a} $\Rightarrow$ \ref{k} in
Theorem~\ref{Thm:C}.
\end{proof}

\begin{proof}[Proof of Theorem~\ref{Thm:A}]
Let $\Omega$, $s$ and
$V$ be top forms as earlier. One can assume they exist since they
always exist locally and the problem is indeed local.
A direct calculation shows that,
for any $u\in \sections{\wedge A}$,
\begin{equation}\label{Eqn:bdiracsquare}
\bdirac^2 u=\big(\thalf(\ld{X_0}+\ld{\xi_0})-\Delta\big)u+
\thalf\big(\thalf\duality{\xi_0}{X_0}-\del X_0\big)u.
\end{equation}

If $(A,\As)$
is a Lie bialgebroid, \ref{k} of Theorem~\ref{Thm:C} together with
\eqref{Eqn:bdiracsquare} above implies that $\bdirac^2$ is the
multiplication by the function
$\fsmile=\thalf(\thalf\duality{\xi_0}{X_0}-\del X_0)$.

Conversely, if $\bdirac^2$ is the multiplication by some function $\fsmile$, \eqref{Eqn:bdiracsquare} implies that the operator $\Delta-\thalf(\ld{X_0}+\ld{\xi_0})$ is the multiplication
by the function \[ g:=\thalf\big(\thalf\duality{\xi_0}{X_0}-\del X_0\big)-\fsmile .\]
From \eqref{enrollment}, it follows that
\begin{multline*} g =\big(\Delta-\thalf(\ld{X_0}+\ld{\xi_0})\big)(1)
=\big(\Delta-\thalf(\ld{X_0}+\ld{\xi_0})\big)(1\cdot 1) \\
=1\cdot \big(\Delta-\thalf(\ld{X_0}+\ld{\xi_0})\big)(1) +
\big(\Delta-\thalf(\ld{X_0}+\ld{\xi_0})\big)(1) \cdot 1
+\ba{1}{\dees 1}-\ba{\dees 1}{1} \\
=2 \big(\Delta-\thalf(\ld{X_0}+\ld{\xi_0})\big)(1) =2 g
.\end{multline*} Hence $g=0$, $\Delta=\thalf(\ld{X_0}+\ld{\xi_0})$
and, by Theorem~\ref{Thm:C}, the pair $(A,\As)$ is a Lie
bialgebroid.

Finally, note that the function $\fsmile=\thalf(\thalf\duality{\xi_0}{X_0}-\del X_0)$
is independent of the order of the pair $(A,\As)$, i.e.
\[ \thalf(\thalf\duality{\xi_0}{X_0}-\del X_0)
=\thalf(\thalf\duality{\xi_0}{X_0}-\dels\xi_0), \]
where
\[ \bdirac^2=\thalf(\thalf\duality{\xi_0}{X_0}-\del X_0)
\qquad \text{and} \qquad
\bdiracs^2=\thalf(\thalf\duality{\xi_0}{X_0}-\dels\xi_0) .\]
Indeed, Corollary~\ref{Cor:blistering} asserts that
$\dels\xi_0=\del X_0$.
\end{proof}

In \cite{MR2103012}, Kosmann-Schwarzbach proved that $\bdirac$ is a deriving operator of the Courant algebroid $A\oplus A^*$. As a consequence, $\bdirac$ is indeed a Dirac generating operator. This is Corollary~\ref{Cor:B}, which is proved below using an argument similar to that in \cite{MR2103012}. 

\begin{proof}[Proof of Corollary~\ref{Cor:B}]
By Theorem~\ref{Thm:A}, we know that $\bdirac^2\in\cinf{M}$.
Thus, to prove that 
\[ \bdirac=\bdees+\bdel=\dees-\del+\thalf(X_0+ \xi_0) \] 
is a Dirac generating operator, we only need to check that it satisfies conditions \ref{conda} and \ref{condb} of Definition~\ref{Def:DiracGTR}.

For all $u\in\sections{A}$, $\theta\in\sections{\As}$
and $v\in\sections{\wedge A}$, one has:
\begin{align*}
\lb{\ds}{f}(v)&=\ds f\wedge v, &
\lb{\del}{f}(v)&=-\inserts_{\dA f} v, \\
\lb{\ds}{u}(v)&=\ds u\wedge v, &
\lb{\del}{u}(v)&=-\LieDer_{u}v+\del u \wedge v, \\
\lb{\ds}{\theta}(v)&=\LieDer_{\theta}v, &
\lb{\del}{\theta}(v)&=\inserts_{\dA \theta}v.
\end{align*}
On the left hand side, the three equalities are trivialities while, on the right hand side, the first two equalities are immediate consequences of \eqref{16} and the third is exactly
\eqref{Eqt:partialinserts}.

A straightforward computation based on the six relations above leads to 
\[ \lb{\dees-\del}{f}(v)=(\ds f+\dA f)\cdot v \] 
and
\[ \lb{\lb{\dees-\del}{u_1+\theta_1}}{u_2+\theta_2}(v)
= \big(\db{(u_1+\theta_1)}{(u_2+\theta_2)}\big)\cdot v ,\]
where $\cdot$ denotes the Clifford action of $\sections{A\oplus A^*}$ on $\sections{\wedge A}$ and $\db{}{}$ is the bracket on $\sections{A\oplus A^*}$ defined by \eqref{4}.

On the other hand, it is obvious that
\[ \lb{e}{f}=0 \qquad \text{and} \qquad \lb{\lb{e}{e_1}}{e_2}=0 ,\]
for all $e,e_1,e_2\in\sections{A\oplus A^*}$ and thus, in particular, for $e=\thalf(X_0+\xi_0)$. 
This completes the proof.
\end{proof}

%
%
%
%

\begin{proof}[Proof of Proposition~\ref{Pro:commissoner}]
While proving Theorem~\ref{Thm:A}, we obtained
\[ \fsmile=\thalf(\thalf\duality{\xi_0}{X_0}-\del X_0) .\]
Therefore, \ref{q} follows from
\begin{align*}
\LieDer_{X_0}(\Omega\otimes s)\otimes V
=& \Omega\otimes \LieDer_{X_0}s\otimes V+\LieDer_{X_0} \Omega\otimes
s\otimes V \\
=& \Omega\otimes \LieDer_{X_0}s\otimes V- \Omega\otimes s\otimes
\LieDer_{X_0} V \\
=& \Omega\otimes\LieDer_{X_0}(s\otimes V)-2\Omega\otimes s\otimes
\LieDer_{X_0}V \\
=& (\duality{\xi_0}{X_0}-2\del X_0)\Omega\otimes s\otimes V
.\end{align*}
The argument for \ref{r} is similar.
\end{proof}

\section{Examples}
\label{last_section}

\subsection{Exact Lie bialgebroids}

Let us briefly recall the notion of an exact Lie bialgebroid
\cite{MR1371234} (see also \cite{0710.3098}).
Let $A$ be a Lie algebroid with bracket $\ba{}{}$ on $\sections{A}$ and anchor map $a:A\to TM$.
Given $\Lambda\in\sections{\wedge^2 A}$ satisfying
$\ba{\ba{\Lambda}{\Lambda}}{X}=0$ for all $X\in\sections{A}$, the bracket
\begin{equation*}\label{Eqt:migrant}
\Lambdabracket{\xi,\theta}=\LieDer_{\Lambdasharp(\xi)}\theta-\LieDer_{\Lambdasharp(\theta)}\xi
-\dA (\Lambda(\xi,\theta))=\LieDer_{\Lambdasharp(\xi)}\theta
-\inserts_{\Lambdasharp(\theta)}\dA \xi
\end{equation*}
on $\sections{\As}$ and the anchor map
$\anchors=\anchor\rond\Lambda\diese$,
make $A^*$ a Lie algebroid.
The pair of Lie algebroid structures on $A$ and $A^*$ fits into a Lie bialgebroid $(A,\As)$, which is known as an exact Lie bialgebroid. If $\ba{\Lambda}{\Lambda}=0$, then $(A,\Lambda)$ is called a Lie algebroid with a Poisson structure and $(A,\As)$ is called a triangular Lie bialgebroid.

Let $\Omega$ be a nowhere zero section of $\sections{\wedge^{\TOP}A^*}$.
We will need the following formula:
\begin{equation}\label{descend}
\dels \theta=-\del \Lambdasharp(\theta)+ 2
\duality{\Lambda}{\dA\theta}, \quad\forall \theta\in \sections{\As}.
\end{equation}
In fact, a simple computation yields that
\[
\LieDer_{\theta}\Omega=\Lambdabracket{\theta,\Omega}=
\LieDer_{\Lambdasharp(\theta)}\Omega+2
\duality{\Lambda}{\dA\theta}\Omega.
\]
Then \eqref{descend} follows from Lemma~\ref{Lem: chores}.

An explicit expression of the modular cocycle $X_0$ of $\As$ was first obtained in \cite{MR2180819} (see also \cites{0710.3098,MR2228680}).
\begin{lem}[\cite{MR2180819}] \label{Lem:combat}
Assume that $(A,\As)$
is an exact Lie bialgebroid as above and $\xi_0$ is the modular
cocycle of $A$, then the modular cocycle of $\As$ is
\[ X_0 = 2 \del \Lambda - \Lambdasharp(\xi_0) .\]
\end{lem}

\begin{prop}
If $(A,A^*)$ is an exact Lie bialgebroid, then $\bdirac^2=0$.
\end{prop}

\begin{proof}
From the proof of Theorem~\ref{Thm:A}, we know that
$\bdirac^2=\fsmile=\thalf(\thalf\duality{\xi_0}{X_0}-\del X_0)$. Then
by Lemma \ref{Lem:combat},
\begin{align*}
\thalf\duality{\xi_0}{X_0}-\del X_0
&=\duality{\xi_0}{\del
\Lambda}-\thalf{\xi_0}{\Lambdasharp(\xi_0)}+\del \Lambdasharp(\xi_0)
-2\del^2\Lambda\\
&=\inserts_{\xi_0}\del\Lambda+\del\inserts_{\xi_0}\Lambda   =
\inserts_{\dA\xi_0}\Lambda, \quad\mbox{( by Lemma~
\ref{Lem:treatises}).}
\end{align*}
Since $\xi_0$ is a cocycle, the result is zero.
\end{proof}

\begin{rmk}
In the case of triangular Lie bialgebroids, the above
result was due to Kosmann-Schwarzbach 
\cite[Theorem 3.3]{MR2103012}. For the Lie bialgebroid $(TM,T^*M)$ associated to a Poisson manifold, this was due to Koszul \cite{MR837203}. In \cite{MR2103012}, Kosmann-Schwarzbach also explained the connection with
\cite{AlekseevXu}.
\end{rmk}

\subsection{Poisson Nijenhuis Lie algebroids}

We recall some basic facts about Poisson-Nijenhuis Lie algebroids (in short, PN-algebroids). Nijenhuis operators and PN-structures are discussed in detail in \cites{MR1077465,MR1421686,MR1614690}. The modular classes of PN-manifolds were studied in
\cites{math/0607784,0710.3098,math/0611202}. The
extension of the modular classes to PN-Lie algebroids was carried out by Caseiro \cite{MR2333510}.

In what follows, we consider   PN-Lie algebroid $(A,N,\Lambda)$,
where $(A,\Abracket{~,~},\rhoA)$ is a Lie algebroid over $M$, $N:A\lon A$ is a Nijenhuis operator and $\Lambda\in \sections{\wedge^2
A}$ is a Poisson structure on $A$.

The Nijenhuis operator $N$ and the Poisson bivector $\Lambda$ are compatible in the following sense:
\begin{equation*}\label{Eqn:waitress}
\Lambda\diese\Ns=N\Lambda\diese \qquad \text{and} \qquad {[\xi,\theta]}_{N\Lambda}=\big(\Lambdabracket{\xi,\theta}\big)_{\Ns},\quad \forall \xi,\theta\in \sections{\As}
.\end{equation*}
Here ${[~,~]}_{N\Lambda}$ is the Lie bracket defined by the bivector field $N\Lambda\in\sections{\wedge^2 A}$ associated to the bundle map $N\Lambda\diese$, and $(\Lambdabracket{~,~})_{\Ns}$ is the Lie bracket obtained from the Lie bracket $\Lambdabracket{~,~}$ by deformation along the Nijenhuis tensor $\Ns$.

The compatibility of $\Lambda$ and $N$ implies
that the bivectors $\Lambda_k$ associated to the bundle maps $\Lambda_k\diese=N^k\rond\Lambda\diese$ ($k\in\mathbb{N}$) are Poisson bivectors and the triples $(A,\Lambda_k,N^l)$ (with $k,l\in\mathbb{N}$) are PN-Lie algebroids. Let us denote the deformation of the Lie algebroid $A$ along $N^l$ by $A_l$,
and the Lie algebroid induced by the Poisson bivector
$\Lambda_k$ by $\As_k$. It was proved that all pairs $(A_l,\As_k)$ are Lie bialgebroids and we have the following fact.
\begin{lem}[\cites{math/0607784,0710.3098,MR2333510}]
Let $\xi_0$ be the modular cocycle of
$A$. Then the modular cocycle of $A_l$ is given by
\begin{equation}\label{Eqt:Scene}
\xi_l=\dA(\trace N^l)+(\Ns)^l \xi_0,
\end{equation}
and one has the identity:
\begin{equation*}\label{Eqt:primaries}
N\del \Lambda_{l-1}-\del \Lambda_l=\tfrac{1}{2l}\Lambdasharp(\dA
(\trace N^l))
.\end{equation*}
\end{lem}

According to Lemma \ref{Lem:combat}, we also know that the modular cocycle of $\As_k$ is given by
\begin{equation}\label{Eqt:legality}
X_k=2\del \Lambda_k-\Lambda_k^\sharp(\xi_0).
\end{equation}

\begin{prop}
For the Lie bialgebroid $(A_l,A_k^*)$, $\bdirac^2=0$.
\end{prop}
\begin{proof}
Note that $N^l:A_l\to A_0$ is a morphism of Lie algebroids and
hence $(\Ns)^l(\xi_0)$ is a 1-cocycle in $A_l^*$:
\[ \dA(\Ns)^l(\xi_0)=(\Ns)^l(\dA\xi_0)=0 .\]
Here $A_0$ is $A$. Thus by Lemma~\ref{Lem:treatises}, one has
\begin{equation}\label{slackers}
\inserts_{(\Ns)^l(\xi_0)}\del=-\del\inserts_{(\Ns)^l(\xi_0)}.
\end{equation}
For the same reason,
\begin{equation}\label{devotion}
\inserts_{\dA \trace N^l}\del=-\del\inserts_{\dA \trace N^l}.
\end{equation}
From (\ref{descend}), we know that the boundary
operator $\dels:\sections{\As_k}\to\cinf{M}$ is given by:
\[ \dels \theta=-\del \Lambda_k^\sharp(\theta)+ 2 \duality{\Lambda_k}{\dA\theta},
\quad\forall \theta\in \sections{\As_k} .\]
Hence by \eqref{Eqt:legality} and \eqref{Eqt:Scene}, one has
\begin{align*}
\thalf\duality{\xi_l}{X_k}-\dels \xi_l
=& \thalf\duality{\dA(\trace N^l)+(\Ns)^l\xi_0}{2\del
\Lambda_k-\Lambda_k^\sharp(\xi_0)} \\
& +\del\Lambda_k^\sharp\big(\dA(\trace
N^l) + (\Ns)^l\xi_0 \big) + 2\duality{\Lambda_k}{\dA (\Ns)^l\xi_0} \\
=& \big(\inserts_{\dA \trace N^l}\del+\del\inserts_{\dA \trace
N^l}+\inserts_{(\Ns)^l\xi_0}\del+\del\inserts_{(\Ns)^l\xi_0}\big) \Lambda_k \\
&+ \thalf\duality{\Lambda_k^\sharp(\dA \trace N^l)}{\xi_0}
- \thalf\duality{\xi_0}{\Lambda_{k+l}^\sharp\xi_0}.
\end{align*}
By \eqref{slackers} and \eqref{devotion}, and using the fact that $\Lambda_{k+l}^\sharp$ is skewsymmetric, we get
\begin{align*}
& \thalf\duality{\xi_l}{X_k}-\dels \xi_l \\
=& \thalf\duality{\Lambda_k^\sharp(\dA \trace N^l)}{\xi_0} \\
=& -\thalf\duality{\Lambdasharp(\dA \trace N^l)}{(\Ns)^k\xi_0} \\
=& l(N\del\Lambda_{l-1}-\del\Lambda_l){(\Ns)^k\xi_0} \quad
\text{(by \eqref{Eqt:primaries})} \\
=& l(\inserts_{(\Ns)^{k+1}\xi_0}\del\Lambda_{l-1}
-\inserts_{(\Ns)^{k}\xi_0}\del\Lambda_{l}) \\
=& -l(\del\inserts_{(\Ns)^{k+1}\xi_0}\Lambda_{l-1}
-\del\inserts_{(\Ns)^{k}\xi_0}\Lambda_{l})
\quad\text{(by \eqref{slackers})} \\
=& -l(\del\inserts_{(\Ns)^{k+l}\xi_0}\Lambda_{l}
-\del\inserts_{(\Ns)^{k+l}\xi_0}\Lambda_{l})
\quad\text{(by \eqref{Eqn:waitress})} \\
=& 0 .\qedhere
\end{align*}
\end{proof}

\subsection{$a+b$ Lie bialgebras}

We finally show an example with $\bdirac^2\neq 0$.
Let $\mfg$ be a 2-dimensional real vector space with base $\{x_1,x_2\}$ and let $\{y_1,y_2\}$ be the dual base of $\mfgs$. Let $\LieG$ be the $a+b$ Lie algebra and let $\LieGs$
be the $c+d$ Lie algebra, i.e.,
\begin{eqnarray*}
\Abracket{x_1,x_2}= a x_1+ bx_2;&& \Asbracket{\eta^1,\eta^2}=  c
\eta^1+d \eta^2.
\end{eqnarray*}
Here $a,b,c$ and $d$ are arbitrary real numbers. It is easily
checked that both $\LieG$ and $\LieGs$ are Lie algebras and the pair $(\LieG,\LieGs)$ is a Lie bialgebra.

Since the modular cocycle $\xi_0\in \LieGs$ of $\LieG$ is given by $\xi_0(x)=\trace(\ad_x)$ for all $x\in\LieG$, one gets $\xi_0=b\eta^1+a\eta^2$. Similarly, the modular cocycle $X_0$ of $\LieGs$ is given by $X_0= d x_1+ c x_2$. Now by Proposition~\ref{Pro:commissoner}, we have
\begin{multline*}
4\fsmile =\duality{4\fsmile(x_1\wedge x_2)}{\eta^1\wedge\eta^2}=\duality{\ld{\xi_0} (x_1\wedge x_2)}{\eta^1\wedge\eta^2} \\
= -\duality{  x_1\wedge x_2 }{\ld{\xi_0}(\eta^1\wedge\eta^2)}
=-(bd+ac).
\end{multline*}




\end{document}